\documentclass[11pt]{amsart}
\usepackage{amsmath, amssymb, graphicx, mathrsfs, enumerate, amsthm, textcomp, url, bbm}
\usepackage[T1]{fontenc}
\usepackage[T5]{fontenc}
\usepackage{mathdots}
\usepackage{comment}
\usepackage[toc,page]{appendix}
\usepackage{hyperref}
\usepackage{tikz-cd}
\hypersetup{
    colorlinks=true,
    linkcolor=blue,
    filecolor=magenta,      
    urlcolor=cyan,
}

\usepackage{color}

\input xy
\xyoption{all}

\allowdisplaybreaks
\usepackage[margin=1in]{geometry}

\numberwithin{equation}{section}
\newtheorem{theorem}{Theorem}[section]
\newtheorem{proposition}[theorem]{Proposition}
\newtheorem{corollary}[theorem]{Corollary}
\newtheorem{lemma}[theorem]{Lemma}

\newtheorem{definition}[theorem]{Definition}

\newtheorem{setting}[theorem]{Setting}
\newtheorem*{remark*}{Remark}

\newtheorem{remark}[theorem]{Remark}

\newcommand{\mfo}{\mathfrak{o}}

\newcommand{\Mfo}{\mathcal{O}}

\newcommand{\gl}{\mathfrak{gl}}

\newcommand{\clb}{\overline{\mathrm{Cl}}}

\newcommand{\End}{\mathrm{End}}
\newcommand{\cL}{\mathcal{L}}
\newcommand{\ce}{\mathcal{E}}
\newcommand{\co}{\mathcal{O}}

\newcommand{\ord}{\mathrm{ord}}
\newcommand*{\Rom}[1]{\expandafter\@slowromancap\romannumeral #1@}
\begin{document}

\title[Global geometrization of local smooth integral models]
    {Global geometrization of local smooth integral models in the Hitchin fibration for $\mathrm{GL}_n$: The Bass case}
\keywords{compactified Jacobian, spectral curve, smoothening, Hitchin fibration}

\subjclass[2020]{MSC 11F72, 14F20, 	14H60}

\author[Sungmun Cho]{Sungmun Cho}
\author[Jungtaek Hong]{Jungtaek Hong}
\thanks{We were supported by  Samsung Science and Technology Foundation under Project Number SSTF-BA2001-04.}

\address{Sungmun Cho \\  Department of Mathematics, POSTECH, 77, Cheongam-ro, Nam-gu, Pohang-si, Gyeongsangbuk-do, 37673, KOREA}

\email{sungmuncho12@gmail.com}

\address{Jungtaek Hong \\ Department of Mathematics, POSTECH, 77, Cheongam-ro, Nam-gu, Pohang-si, Gyeongsangbuk-do, 37673, KOREA}

\email{jungtaekhong123@gmail.com}

\maketitle

\begin{abstract}
In previous joint work, we proposed a new method to study local orbital integrals for $\mathrm{GL}_n$ (where $n=3$ or in the Bass case), by employing a smoothening method of a certain scheme defined over a henselian ring.
In this paper, we geometrize this local smoothening method within the framework of the global Hitchin fibration for $\mathrm{GL}_n$ in the Bass case. 
Consequently, we provide a closed formula for the $\ell$-adic cohomology of the compactified Jacobian of a spectral curve over a finite field with double singularities whose local rings are integral domains, provided that the normalization of a spectral curve is isomorphic to $\mathbb{P}^1$.
    \end{abstract}

\tableofcontents

\section{Introduction}


The (local) orbital integral is the geometric side of \textit{the celebrated Arthur-Selberg trace formula}, which is a fundamental object for understanding endoscopic transfer. 
This has been studied by many mathematicians, such as Ngô Bảo Châu, using the Hitchin fibration.
The latter is a global object on a curve over a finite field. 
On the other hand, the authors' former works (\cite{CKL} and \cite{CHL}) proposed another geometric method, called \textit{smoothening}. 
 This involves a local object used to desingularize a certain scheme defined over a henselian ring of arbitrary characteristic. 
 The main goal of this paper is to compare these two methods (for the Bass case in $\mathrm{GL}_n$) and to derive an explicit formula for the $\ell$-adic cohomology of the compactified Jacobian of a spectral curve $Y_\chi$ (when the normalization of $Y_\chi$ is isomorphic to $\mathbb{P}^1_k$) having double singularities whose local rings $\Mfo_{Y_{\chi,x}}$ are integral domains. 

\subsection{Global object: Hitchin fibration for $\mathrm{GL}_n$} 
For a smooth projective curve $C$ with  a large effective divisor $D$ defined over a finite field $k$, 
 the Hitchin moduli stack $\mathbb{M}$ parametrizes the vector bundles $\mathcal{E}$ on $C$ of rank $n$ along with a morphism $\theta:\mathcal{E}\rightarrow \mathcal{E}(D)$.
The Hitchin fibration $\Phi$ is defined to be the "characteristic polynomial" morphism from  $\mathbb{M}$ to the affine space $\mathbb{A}$, which parametrizes the coefficients of the characteristic polynomial $\chi$ of an endomorphism.
    By Ngô's product formula(cf. Proposition \ref{prop:prod}), $\Phi^{-1}(\chi)(k)$ is understood as the product of the affine Springer fiber at each closed points of $C$, modulo certain equivalence.
Consequently the point count of $\Phi^{-1}(\chi)(k)$ is (a finite product of) the orbital integral for $\mathrm{GL}_n$ corresponding to $\chi$.

Ngô proved the support theorem for the Hitchin fibration.
This is a cohomological statement that allows us to deduce the fundamental lemma solely by studying the regular locus of $\mathbb{A}$, without calculating the cohomology of $\Phi^{-1}(\chi)$ and thus of the orbital integral.
Over the regular and elliptic locus, the Hitchin fibers are abelian varieties but in general they are not.
On the other hand, $\Phi^{-1}(\chi)$ has a geometric moduli description described below. 
\begin{theorem}(\cite[Section 4.4.1]{Ngo10})\label{thm:spectralcorrintro}
    $\Phi^{-1}(\chi)$ is isomorphic to the \textbf{compactified Jacobian} $\overline{\mathrm{Pic}^0}(Y_\chi)$, which parametrizes the rank-1 torsion-free sheaves on a curve $Y_\chi$ depending on $\chi$, called the \textbf{spectral curve}.
\end{theorem}

Thus, an explicit formula for the cohomology of $\Phi^{-1}(\chi)$  carries 
rich geometric information, not only for endoscopic transfer.

\subsection{Local object: Smoothening}
To explicitly compute the (local) orbital integral for $\chi$, \cite{CKL} and \cite{CHL} introduced a method called \textit{smoothening}, which provides another geometric interpretation of the orbital integral in a purely local setting.
The key idea is to use geometric measures to reinterpret the orbital integral.
Choose $\gamma\in \mathfrak{gl}_n(\mfo)$ where $\mfo$ is the ring of integers $\mfo$ of a non-Archimedean local field, so that its characteristic polynomial $\chi$ has coefficients in $\mfo$.
Consider a morphism $\varphi: \mathfrak{gl}_{n,\mfo}\rightarrow \mathbb{A}_{\mfo}^n$ over $\mfo$  sending a matrix to the coefficients of the characteristic polynomial.
Then the orbital integral is reinterpreted as the volume of $\varphi^{-1}(\chi)(\mfo)$.
If $\varphi^{-1}(\chi)$ is smooth, then Weil's volume formula easily yields the orbital integral:
\begin{theorem}\cite[Theorem 2.2.5]{Weil}
    If $\varphi^{-1}(\chi)$ is a smooth scheme over $\mfo$, then the orbital integral is $\frac{\#(\varphi^{-1}(\chi)(k))}{q^{n^2-n}}$, where $q$ is the cardinality of the residue field $k$ of $\mfo$.
\end{theorem}
However, $\varphi^{-1}(\chi)$ is not smooth for most $\chi$ of interest.  
To resolve this issue, \cite{CKL} (for $n=3$) and \cite{CHL} (in the Bass case) stratified $\varphi^{-1}(\chi)(\mfo)$ into a disjoint union of  sets of $\mfo$-points of smooth schemes over $\mfo$, so that the volume of each stratum is obtained using Weil's formula.

We briefly explain the stratification in the case that $\mfo[X]/(\chi(X))$ is a \textit{Bass order} and an integral domain, which is  treated in \cite{CHL}.
In this case, stratification of $\varphi^{-1}(\chi)(\mfo)$ is indexed by a pair $T=(\mathcal{F},l)$ called a \textit{type}, which consists of a finite-dimensional $k$-algebra $\mathcal{F}$ and an integer $l$.
Viewing a matrix $m$ in $\varphi^{-1}(\chi)(\mfo)$ as an endomorphism of a lattice $L\cong \mfo^n$, $\mathcal{F}$ represents the cokernel of $m$ up to isomorphism, and $l$ governs the size of the Jordan blocks of $\overline{m}$, the reduction of $m$ modulo a uniformizer of $\mfo$.
Then, the collection of matrices sharing the same type $T$ is functorially extended to form a representable subfunctor $(\varphi^T)^{-1}(\chi)$ of $\varphi^{-1}(\chi)$ as  an fppf-sheaf on $\mfo$, which turns out to be a smooth scheme over $\mfo$. 
A more detailed process of smoothening for Bass orders is explained in Section \ref{subsec:CHL}.

Here a Bass order is an order of a number field or a local field whose fractional ideals are generated by two elements. This often appears in algebraic number theory, algebraic geometry, and topology. 
In algebraic geometry, the local ring $\Mfo_{Y_{\chi,x}}$ of a spectral curve $Y_\chi$ at a singular point $x$ is a Bass order if and only if $x$ is a double point(cf. Proposition \ref{prop:Bass}). 
We refer to \cite[Introduction]{CHL} for further details on the background of Bass orders. 
For the rest of Introduction, we will suppose that the local ring $\Mfo_{Y_{\chi,x}}$ of a spectral curve $Y_\chi$ at any singular point $x$ is a Bass order (i.e. singularities of $Y_\chi$  are double points) and an integral domain.

\subsection{Comparison: algebraic and geometric stratification}
It is natural to ask whether the local smoothening process in the Bass case is applicable to the framework of the Hitchin fibration to yield a smooth stratification of the Hitchin fiber $\Phi^{-1}(\chi)$.
To globalize the stratification given in the previous subsection based on local matrix-theoretic arguments, 
 we use the adelic description of $\mathbb{M}(R)$ for a $k$-algebra $R$, as stated in Remark \ref{rmk:adelic}.
We define a global type $T$ (cf. Definition \ref{def:typeofchia}) for an adelic description of $\mathbb{M}(R)$  to reflect each local type at a point of a curve $C$. 
For each type $T$, we construct the restriction $\Phi^T:\mathbb{M}^T\rightarrow \mathbb{A}^T$ of $\Phi:\mathbb{M}\rightarrow \mathbb{A}$, called \textit{the restricted Hitchin fibration}. 
This  yields the following \textbf{algebraic stratification}.


\begin{theorem}[Corollary  \ref{cor:algstraoff-1a}]
    Let $\mathcal{I}$ be the finite set of types $T$ of $\chi$.
    Then, for $T\in \mathcal{I}$, $(\Phi^T)^{-1}(\chi)$ is a locally closed and smooth subvariety of $\Phi^{-1}(\chi)$.
    Furthermore, for a finite field extension $k'/k$ or for $k'=\bar{k}$, we have the following stratification of $\Phi^{-1}(\chi)(k')$:
    $$\Phi^{-1}(\chi)(k')=\bigsqcup\limits_{T\in \mathcal{I}}(\Phi^T)^{-1}(\chi)(k').$$
\end{theorem}


On the other hand, there is a well-known stratification of the compactified Jacobian $\overline{\mathrm{Pic}^0}(Y_\chi)$ when the stalk $\Mfo_{Y_{\chi,x}}$ of a spectral curve $Y_\chi$ at any singular point $x$ is a Bass order and an integral domain, due to \cite{Gau97}. 
We call it the \textbf{geometric stratification}, which is stated below:

\begin{proposition}\cite[Proposition 3.4]{Gau97}(or Proposition \ref{prop:geomstrfork0})\label{prop14}
    For a field extension $k'/k$ where $k'$ is a finite extension or $k'=\bar{k}$, we have the following stratification of $\Phi^{-1}(\chi)(k')$:
    $$\Phi^{-1}(\chi)(k')= \bigsqcup_{\pi':Y_\chi'\rightarrow Y_\chi}\mathrm{Pic}^0(Y_\chi')(k'),$$
    where the index runs over the set of finite surjective birational maps $\pi':Y_\chi'\rightarrow Y_\chi$ up to isomorphism, and $\mathrm{Pic}^0(Y_\chi')$ is the generalized Jacobian variety which parametrizes line bundles of degree $0$ on $Y_\chi'$.
    Here, we use Theorem \ref{thm:spectralcorrintro} to identify $\Phi^{-1}(\chi)$ with $\overline{\mathrm{Pic}^0}(Y_\chi)$.
\end{proposition}
The indexing maps $\pi':Y_\chi'\rightarrow Y_\chi$ are called \textit{partial normalizations} of $Y_\chi$.
We prove that the algebraic and geometric stratifications are the same, as stated in the following:

\begin{theorem}\label{thm:eqoftwostraintro}[Theorem \ref{thm:eqoftwostra}]
    Let $\pi':Y_\chi'\rightarrow Y_\chi$ be a partial normalization of $Y_\chi$.
    Then there exists a unique type $T$ of $\chi$ such that
    $$(\Phi^T)^{-1}(\chi)= \mathrm{Pic}^0(Y_\chi')$$
    with respect to the identification between $\Phi^{-1}(\chi)\cong \overline{\mathrm{Pic}^0}(Y_\chi)$ in Theorem \ref{thm:spectralcorrintro}.
\end{theorem}
Here $Y_\chi'$ is not necessarily a spectral curve. Due to this theorem, $\mathrm{Pic}^0(Y_\chi')$ is described as the fiber of a restricted Hitchin fibration. 
We will briefly outline a main structure of the proof. 

The proof starts from the observation that a partial normalization of $Y_\chi$ is completely determined by the data of \textit{overorders} of $\Mfo_x[X]/(\chi_x(X))$ for each $x\in C$.
An overorder of $\Mfo_x[X]/(\chi_x(X))$ is an order lying between it and its normalization.
These overorders have an explicit description given in \cite{CHL}. Note that it is based on an explicit formula for the local orbital integral of $\mathrm{GL}_n$ in the Bass case.
Using this description, we prove the theorem by showing that the data of overorders also completely determines the matrices in $\Phi^{-1}(\chi)$ that share the same given type, which yields the agreement between the two strata.

\begin{remark}
\begin{enumerate}
    \item 
The geometric stratification of \cite[Proposition 3.4]{Gau97} (see Proposition \ref{prop14}) fails when $Y_\chi$ admits a singularity that is not a double point, due to the definition of a Bass order (cf. \cite[Proposition 6.3]{CHL}). 
However, we expect an algebraic stratification to exist once the local smoothening process is established (for example, when $n=3$).
We will pursue this direction in future work.

\item A partial normalization $Y_\chi'$ is not necessarily a spectral curve, as the stalk of $Y_\chi'$ at a point is not necessarily a simple extension of the stalk of $C$ (cf. Proposition \ref{thm:overorders}). 
Nonetheless, Theorem \ref{thm:eqoftwostraintro} implies that $\mathrm{Pic}^0(Y_\chi')$ can still be described within the framework of a restricted Hitchin fibration (cf. Diagram \eqref{diag:f'}), and that each point of $\mathrm{Pic}^0(Y_\chi')(k)$ has an adelic description (cf. Remark \ref{rmk:ademf} for $n=2$ and Remark \ref{rmk:ademfly} for $n\geq 3$). 
We refer to Remark \ref{rmk:nouse} for further discussion. 
\end{enumerate}
\end{remark}


\subsection{Explicit formula for $H_c^i(\overline{\mathrm{Pic}^0}(Y_\chi)_{\bar{k}},\mathbb{Q}_\ell)$ in the Bass case}
In the case where the normalization of $Y_\chi$ is isomorphic to $\mathbb{P}^1_k$, by using a similar argument to the proof of Theorem \ref{thm:eqoftwostraintro}, we show that for each type $T$, $(\Phi^T)^{-1}(\chi)$ is an affine space of dimension explicitly given in Lemma \ref{lem:affine}.
Thus, the algebraic (or geometric) stratification yields an affine paving of the Hitchin fiber $\Phi^{-1}(\chi)$, leading to the formula for its cohomology.


\begin{theorem}[Theorem \ref{thm:cohomology}]
    Let $\mathcal{I}$ be the set of types of $\chi$ (cf. Definition \ref{def:typeofchia} with $n\geq 3$ and Remark \ref{rmk:typeforn=2} with $n=2$).
If the normalization of $Y_\chi$ is isomorphic to $\mathbb{P}^1_k$, then $H_c^i(\Phi^{-1}(\chi)_{\bar{k}},\mathbb{Q}_\ell)$ has the following decomposition as a $\mathrm{Gal}(\bar{k}/k)$-module:  
    $$H_c^i(\Phi^{-1}(\chi)_{\bar{k}},\mathbb{Q}_\ell)\left(\cong  H_c^i(\overline{\mathrm{Pic}^0}(Y_\chi)_{\bar{k}},\mathbb{Q}_\ell)\right) \cong \bigoplus \limits_{T\in \mathcal{I}}H_c^i((\Phi^T)^{-1}(\chi)_{\bar{k}},\mathbb{Q}_\ell),$$
 where for a type $T=(\mathcal{F},(l_x)_{x\in \mathcal{S}^{\mathcal{F}}_1})$,
    $$H_c^i((\Phi^T)^{-1}(\chi)_{\bar{k}},\mathbb{Q}_\ell)\cong \begin{cases}
        \mathbb{Q}_\ell(-\frac{i}{2})&\textit{ if $n\geq 3$ and $i=2(\sum\limits_{x\in \mathcal{S}'}(\frac{n-3}{2}-l_x))$};\\
        \mathbb{Q}_\ell(-\frac{i}{2})&\textit{ if $n=2$ and $i=\sum\limits_{x\in\mathcal{S}'}(b_x-a_x-1)=2(\sum\limits_{x\in\mathcal{S}'}(\frac{2d_x+\ord_x(a_2)-1}{2}-a_x))$};\\
        0&otherwise.
    \end{cases}  $$
\begin{itemize}
\item Here $D=\sum\limits_{x} d_x[x]$ and $\chi(X)=X^n-a_1X^{n-1}+\dots +(-1)^na_n$ with $a_i\in H^0(C, \co_C(iD))$;
\item $\mathcal{S}'$ is given in \eqref{eq:defofs'} (see Setting \ref{settings}.(2) for $\mathcal{S}$) and $\mathcal{S}_1^\mathcal{F}$ is given in Definition \ref{def:typeofchia}.(2);
\item $\mathcal{F}$ is described in Remark \ref{rmk:rx}.(1) (cf. Definition \ref{def:coherentsheafF});
\item when $n\geq 3$,  $n$ is odd (cf. Proposition \ref{thm:overorders}.(1)) and $l_x \in \mathbb{Z}$ for $x\in \mathcal{S}_1^\mathcal{F}$ such that $0\leq l_x \leq n/2-1$ (cf. Definition \ref{def:typeofchia}), where   we let $l_x:=-1$ if $x\in \mathcal{S}'\backslash\mathcal{S}^{\mathcal{F}}_1$;
    \item when $n=2$, $\ord_x(a_2)$ is odd (cf. Proposition \ref{thm:overorders}.(2)) and    $0\leq a_x\leq b_x \in \mathbb{Z}$   such that $a_x+b_x=2d_x+\ord_x(a_2)$ (cf. Lemma \ref{lem:affine}). 
\end{itemize}
\end{theorem}

Using a precise description of the set of types $\mathcal{I}$ in Remark \ref{rmk:rx}.(1), we derive an explicit closed formula of the above theorem below.

\begin{corollary}[Corollary \ref{cor:explicitcohomology}]
Suppose that the normalization of $Y_\chi$ is isomorphic to $\mathbb{P}^1_k$.
Let $\mathcal{S}'=\{x_1, \dots, x_t\}$ so that $\#\mathcal{S}'=t$, the number of singularities of $Y_\chi$ (cf. \eqref{eq:emphas'}).  Let
\[
\begin{cases}
        n=2m+1, ~P_m(X)=1 + X + X^2 + \dots + X^m  &\textit{ if $n\geq 3$};\\
       2d_{x_j}+\ord_{x_j}(a_2)=1+2m_j\left(\geq 3\right), ~Q_j(X) = 1 + X + X^2 + \dots + X^{m_j} &\textit{ if $n=2$}. 
    \end{cases}
\]
Then for any integer $i$,  $H_c^{2i+1}(\overline{\mathrm{Pic}^0}(Y_\chi)_{\bar{k}},\mathbb{Q}_\ell)=0$ and any direct summand of $H_c^{2i}(\overline{\mathrm{Pic}^0}(Y_\chi)_{\bar{k}},\mathbb{Q}_\ell)$ is isomorphic to $\mathbb{Q}_\ell(-i)$ (if it is not zero).
In addition, the Poincar\'e polynomial is 
$$\begin{cases}
\left(P_m(X^2)\right)^t  &\textit{ if $n\geq 3$};\\
Q_1(X^2) Q_2(X^2) \dots Q_t(X^2) &\textit{ if $n=2$}. 
    \end{cases}$$
\end{corollary}

We refer to Corollary \ref{cor:explicitcohomology} for a numerical description of the above formula.

\begin{remark}
The factorization of the Poincaré polynomial in the above corollary can be understood via the Künneth formula on affine Springer fibers; we refer to Remark \ref{rmk:kun} for details.
\end{remark}


\textbf{Organization.}
In Section \ref{sec:localsmoothening}, we summarize the local smoothing arguments of \cite{CKL} and \cite{CHL}. 
Building on this, Section \ref{section3} introduces the Hitchin fibration $\Phi$ and globalizes the smoothening process; by defining the type $T$ and the strata $(\Phi^T)^{-1}(\chi)$, we establish the algebraic stratification.
Section \ref{Section:geomstr} defines the spectral curve and provides a geometric description of the Hitchin fibration. We conclude this section by constructing the geometric stratification and proving its equivalence to the algebraic one.
Finally Section \ref{sec:dec} derives an explicit formula for the $\ell$-adic cohomology of the compactified Jacobian of a spectral curve over a finite field with double singularities whose local rings are integral domains.

\textbf{Acknowledgement.} 
We sincerely thank Wansu Kim and  Sug Woo Shin for pointing out an error in Lemma \ref{lem:repre}, Zhiwei Yun for suggesting an affine paving in Theorem \ref{thm:cohomology}, Zongbin Chen for helpful discussions on the purity theorem for affine Springer fibers, Lucien Hennecart for an advice regarding the topology on the stack of coherent sheaves, and Kyoung-Seog Lee for recommending excellent sources on stacks.
We also appreciate Ngô Bảo Châu for encouraging this project at the Samsung Global Research Symposium 2025 and Yuchan Lee for meaningful and ongoing  discussions.

\section{Revisit to the local smoothening process}\label{sec:localsmoothening}
In this section, we review the smoothening process described in \cite{CKL} and \cite{CHL}. 
Let us start with the notations in the local setting.
\begin{itemize}

    \item Let $k$ be a finite field extension of $\mathbb{F}_p$ or an algebraic closure of $\mathbb{F}_p$, and let $\bar{k}$ be an algebraic closure of $k$.

\item Let $\mfo$ be the ring of integers in a non-Archimedean local field $F$ whose residue field is $k$. We choose a uniformizer $\pi$ in $\mfo$.
We write $\ord(a)$ with $a\in F$ for the exponential order of $a$.

\item For $a\in A$ or $\psi(X) \in A[X]$ with a flat $\mfo$-algebra $A$, $\overline{a}\in A \otimes_{\mfo} k$ or $\overline{\psi(X)}\in A\otimes_{\mfo} k[X]$ is the reduction of $a$ or $\psi(X)$ modulo $\pi$, respectively.

\item Let $\mathfrak{gl}_{n,A}$ be the Lie algebra of $\mathrm{GL}_{n,A}$ which is the general linear group of dimension $n^2$ defined over $A$, where  $A$ is an $\mfo$-algebra;

\item Let $\mathrm{M}_n(A)$ be the set of $n \times n$ matrices with entries in $A$;

\item Let $\mathbb{A}^{n}_A$ be the affine space of dimension $n$ defined over $A$;
\item Let $\chi_{m}(X)\in R[X]$  be the characteristic polynomial of $m \in \mathfrak{gl}_{n, \mfo}(R)$ for a flat $\mfo$-algebra $R$;

\item  Define the following morphism of schemes defined over $\mfo$: 
$$\varphi_n: \mathfrak{gl}_{n, \mfo} \longrightarrow \mathbb{A}_{\mfo}^n,   ~~~~~~~  m \mapsto 
(a_{1}, \dots, a_n) \left(=\textit{coefficients of }\chi_{m}(X)\right),$$ 
where $\chi_{m}(X)=X^n-a_1X^{n-1}+\dots + (-1)^{n-1}a_{n-1}X+(-1)^na _n$.

\item Choose an elliptic (so regular) and semisimple element $\gamma \in \mathfrak{gl}_{n,\mfo}(\mfo)$. Equivalently, $\chi_{\gamma}(X)$ is irreducible over $\mfo$.
Let $d:=\ord(a_n)$. 
\end{itemize}

The orbital integral for a regular and semi-simple element $\gamma$ is defined as the volume of $\varphi_n^{-1}(\chi_{\gamma})(\mfo)$ with respect to the geometric measure. We refer to \cite[Section 2]{CKL} for a more precise explanation.
If the scheme $\varphi_n^{-1}(\chi_{\gamma})$ is smooth over $\mfo$, then its volume can be easily computed. In general, the scheme is far from smooth, which makes it hard to investigate. 
A main result of \cite{CKL} and \cite{CHL} is to give a precise description of a smooth integral model of $\varphi_n^{-1}(\chi_{\gamma})$, which will be summarized in this section.







\subsection{Revisit to the smoothening process in \cite{CKL}}\label{subsec:CKL}
The following is taken from \cite[Section 3.2]{CKL}:
\begin{itemize}
\item Let $L$ be a free $\mfo$-module of rank $n$ and let $M\subset L$ be a submodule (thus free) of rank $n$.

\item Define a functor $\cL(L,M)$ on the category of flat $\mfo$-algebras to the category of sets such that 
\[
\cL(L,M)(R)=\{f\mid\textit{$f:L\otimes R\rightarrow M\otimes R $ is $R$-linear and surjective}\}
\]
for a flat $\mfo$-algebra $R$.
The functor $\cL(L,M)$ is then represented by an open subscheme of $\mathrm{Hom}_{\mfo}(L,M)$, which is an affine space over $\mfo$ of dimension $n^2$, so as to be smooth over $\mfo$.
Here, `open subscheme' structure is due to surjectivity in $\cL(L,M)(R)$. 

\item  We can assign the $\mfo$-scheme structures on 
$\mathrm{End}(L)(\mfo)$ by defining the functor on the category of flat $\mfo$-algebras to the category of  sets as follows:
\[
\End(L)(R)=\End_{R}(L\otimes R) ~~  \textit{  for an $\mfo$-algebra $R$.}
\]

We then express $\gl_{n, \mfo}=\End(L)$ so that $\cL(L,M)(\mfo)$  is  an open subset  of $\gl_{n, \mfo}(\mfo)$. 



\end{itemize}

\begin{definition}{\cite[Definition 3.5]{CKL}}\label{def_type}
The type of $M$ is defined to 
    be $(k_1,\dots,k_n)$ where $k_i\in \mathbb{Z}_{\geq 0}$ and $ k_i\leq k_j$ for $i\leq j$, such that 
    \[
    L/M\cong \mfo/(\pi^{k_1})\oplus\mfo/(\pi^{k_2})\oplus \dots \oplus \mfo/(\pi^{k_n}).
    \]
For example, if $L$ (resp. $M$) is spanned by $(e_1, e_2)$ (resp. $(e_1, \pi \cdot  e_2)$), then the type of $M$ is $(0, 1)$.
\end{definition}




Choose $M$ such that  $type(M)=(0, \dots, 0, d)$. 
We define another functor $\mathbb{A}_{M}$ on the category of flat $\mfo$-algebras to the category of sets such that 
$\mathbb{A}_{M}(R)=R^{n-1}\times  
\pi^d R$
for a flat $\mfo$-algebra $R$. Then $\mathbb{A}_{M}$ 
is represented
by an affine space over $\mfo$ of dimension $n$.
Define a morphism 
$$\varphi_{M}: \cL(M)(R) \longrightarrow \mathbb{A}_{M}(R), ~~~~~~~~~ m \mapsto \textit{coefficients of }\chi_{m}(X). $$ 
Here we emphasize that $d$ is allowed to be $0$ so that $M=L$, which happens when the reduction of $\chi_{\gamma}(X)$ in $k[X]$ is irreducible.
We refer to \cite[Remark 4.3]{CKL} for further details about $\varphi_{M}$.

\begin{theorem}{\cite[Theorem 4.6]{CKL}}\label{thm:localtype1line}
The morphism $\varphi_{M}$ is smooth at any point contained in  $\varphi_{M}^{-1}(\chi_{\gamma})(\bar{k})$.
\end{theorem}

By \cite[Remark 7.4]{GY} and \cite[Section 3.4.2]{CKL}, we generalize the above result as follows, which will be necessary for a global smoothness in Section \ref{section3}.
For any integer $t$, let 
\begin{equation}\label{eq:localhitch}
    \varphi_{M}^t: \pi^t \cdot \cL(L,M)(R) \longrightarrow \mathbb{A}_{M}^t(R),  ~~~~~~~~~ \pi^t\cdot m \mapsto \textit{coefficients of }\chi_{\pi^t\cdot m}(X),
\end{equation}
where $\mathbb{A}_{M}^t(R):= \left(\bigoplus\limits_{i=1}^{n-1}\pi^{i\cdot t}R\right) \oplus \pi^{n\cdot t+d}R$ for a flat $\mfo$-algebra $R$.

\begin{corollary}\label{cor:genlocalthm1}
The morphism $\varphi_{M}^t$ is smooth at any point contained in  $\left(\varphi^t_{M}\right)^{-1}(\chi_{\pi^t\cdot \gamma})(\bar{k})$.
\end{corollary}
\begin{proof}
The proof directly follows from the commutative diagram below (cf. \cite[Diagram (3.9)]{CKL}):
\begin{equation}\label{diag48}
\xymatrixcolsep{5pc}\xymatrix{
\pi^t \cdot \cL(L,M) \ar[d]^{\iota_1, \cong} \ar[r]^{\varphi_{M}^t} &  \mathbb{A}_{M}^t\ar[d]^{\iota_2, \cong} \\
\cL(L,M) \ar[r]^{\varphi_{M}} &  \mathbb{A}_{M}
}
\end{equation}
Here both $\iota_1$ and $\iota_2$ are isomorphisms of $\mfo$-schemes described below: \[
\left\{
  \begin{array}{l }
\iota_1: \pi^t \cdot \cL(L,M)\cong \cL(L,M), ~~~~~~~ \pi^t\cdot m \mapsto m;\\
\iota_2: \mathbb{A}_{M}^t \cong \mathbb{A}_{M}, ~~~~~~~ (\pi^t a_1, \dots, \pi^{n\cdot t} a_n) \mapsto (a_1, \dots, a_n).
    \end{array} \right.
\]
\end{proof}





\subsection{Revisit to the smoothening process in \cite{CHL}}\label{subsec:CHL}

\subsubsection{The construction of smoothening in Bass case of \cite{CHL}}\label{subsec:revisitchl}

Choose $M$ with  $type(M)=(0, \dots, 0, 1, 1)$ and choose a basis $(e_1, \dots, e_n)$ for $L$ such that $(e_1, \dots, \pi e_{n-1}, \pi e_n)$ is a basis for $M$.

Consider two integers $l_1$ and $l_2$ such that $l_1+l_2=n-2$ with $0\leq l_1\leq l_2$.
For each $l_1$, choose a matrix $X_{l_1}\in \mathrm{M}_{n-2}(\mfo)$ whose reduction modulo $\pi$ is similar to $J_{l_1}\perp J_{l_2}$, where $J_{l_i}$ is the Jordan canonical form of size $l_i$ with diagonal entries $0$ and with superdiagonal entries $1$.

Define a functor $\mathrm{End}_\mfo(L)_{M, l_1}$ on the category of flat $\mfo$-algebras to the category of sets such that 
\[
\mathrm{End}_\mfo(L)_{M, l_1}(R)=\left\{m=\begin{pmatrix}
X_{l_1}+\pi X^\dag_{1,1} &X_{1,2} & X_{1,3} \\
\pi X_{2,1} &\pi X_{2,2} & \pi X_{2,3} \\
\pi X_{3,1} &\pi X_{3,2} & \pi X_{3,3}
\end{pmatrix}\in \cL(L,M)(R) 
\right\}
\]
for a flat $\mfo$-algebra $R$, where $X^\dag_{1,1}\in \mathrm{M}_{n-2}(R)$ is of size $n-2$ and so on.
The functor $\mathrm{End}_\mfo(L)_{M, l_1}$ is then represented by an open subscheme of the affine space over $\mfo$ of dimension $n^2$.
Note that the functor $\mathrm{End}_\mfo(L)_{M, l_1}$ depends on the choice of $X_{l_1}$, but two different choices give an isomorphism given by a conjugation.


We define another functor $\mathbb{A}_{l_1}$  on the category of flat $\mfo$-algebras to the category of sets such that
$\mathbb{A}_{l_1}(R)=(\pi R)^{l_2+1}\times (\pi^2 R)^{n-l_2-1}$
for a flat $\mfo$-algebra $R$. Then $\mathbb{A}_{l_1}$ is represented by an affine space over $\mfo$ of dimension $n$. 
Define a morphism $\varphi_{l_1}$ with $l_1\geq 0$ as follows;
\begin{equation}\label{eq:localhitchbass}
\varphi_{l_1}: \mathrm{End}_\mfo(L)_{M, l_1} \longrightarrow \mathbb{A}_{l_1}, ~~~~~ m \mapsto  \textit{coefficients of }\chi_m(X).
\end{equation}
  
We refer to \cite[Section 4.3.1 and Remark 4.6]{CHL} for further details about $\varphi_{l_1}$.

\begin{theorem}{\cite[Proposition 4.7]{CHL}}\label{thm:localBassmain}
If $n\geq 4$ is even such that $X^2+\overline{1/\pi_x\cdot a_{n/2}}\cdot X+\overline{1/\pi_x^2\cdot a_{n}} \in k[X]$ is irreducible, or if $n\geq 3$ is odd, then  $\varphi_{l_1}$ is smooth at any point contained in  $\varphi_{l_1}^{-1}(\chi_{\gamma})(\bar{k})$.
\end{theorem}

As in Corollary \ref{cor:genlocalthm1}, the above result is generalized as follows.
For any integer $t$, let 
\begin{equation}\label{eq:localhitchinbasst}
\varphi_{l_1}^t: \pi^t \cdot \mathrm{End}_\mfo(L)_{M, l_1}(R) \longrightarrow \mathbb{A}_{l_1}^t(R),  ~~~~~~~~~ \pi^d\cdot m \mapsto \textit{coefficients of }\chi_{\pi^t\cdot m}(X),
\end{equation}
where $\mathbb{A}_{l_1}^t(R):= \left(\bigoplus\limits_{i=1}^{l_2+1}\pi^{i\cdot t+1}R\right) \oplus \left(\bigoplus\limits_{i=l_2+2}^{n}\pi^{i\cdot t+2}R\right)$ for a flat $\mfo$-algebra $R$.

\begin{corollary}\label{cor:genlocalthm}
The morphism $\varphi_{l_1}^t$ is smooth at any point contained in  $\left(\varphi^t_{l_1}\right)^{-1}(\chi_{\pi^t\cdot \gamma})(\bar{k})$.
\end{corollary}
\begin{proof}
The proof is parallel to that of Corollary \ref{cor:genlocalthm1} and thus we skip it.
\end{proof}

\subsubsection{Bass orders in a non-archimedean local field}\label{subsec:bass}
In this subsection, we will review the notion of Bass order, which is the main subject in \cite{CHL}. 
The above smoothening processes yield a formula for orbital integrals when $\chi_\gamma(X)$ determines a Bass order $\mfo[X]/(\chi_\gamma(X))$. 
Since we will geometrize these smoothening processes in Section \ref{section3} and will use it in geometric stratification in Section \ref{Section:geomstr}, we will review a Bass order below.
Note that a Bass order is defined for both global and local number fields, but we will work on its local nature in this paper. 
We refer to \cite[Section 3]{CHL} for a detailed explanation,  to \cite[Proposition 3.5]{CHL} for a local classification, and to \cite[Theorem 1.2]{CHL} for a global classification.

 
 \begin{definition}\cite[Part 1, Notations and Definition 3.1]{CHL}\label{def:bass}
Let $E$ be a finite field extension of $F$ with $\Mfo_E$ its ring of integers. Let $\Mfo$ be an order containing $\mfo$ in $E$, that is, $\Mfo$ is a subring of $\Mfo_E$ such that $\Mfo\otimes_\mfo F\cong E$. 
    \begin{enumerate}
\item $\Mfo$ is called a Bass order if every ideal of $\Mfo$ is generated by at most two elements. 
\item A \textit{fractional $\Mfo$-ideal} $I$ is a finitely generated $\Mfo$-submodule of $E$ such that $I\otimes_{\Mfo}F=E$.
As an $\mfo$-module, every fractional $R$-ideal is free of rank $n$.

\item The ideal class monoid of $\Mfo$, denoted by $\clb(\Mfo)$, is the set of fractional $\Mfo$-ideals up to multiplication by principal fractional $\Mfo$-ideals. We use $[I]$ for an element of $\clb(\Mfo)$, where $I$ is a fractional $\Mfo$-ideal.  
     \item An overorder of $\Mfo$ is an order in $E$ which contains $\Mfo$.
     \end{enumerate}
 \end{definition}
 \begin{proposition}\label{prop:overorder}
 For a Bass order $\Mfo$,   $[I]=[(I:I)]$ in $\clb(\Mfo)$, where $(I:I):=\{x\in E\mid xI\subset I\}$ is an overorder of $\Mfo$ which is independent of the choice of $I$.
 \end{proposition}
 \begin{proof}
     It is straightforward to check that $(I:I)$ is a subring of $E$ which contains $\Mfo$.
     Therefore, it is an order.
     By \cite[Proposition 1.1.11]{DTZ}, $(I:I)$ is independent from the choice of $I$ in $[I]$.
     Since $\Mfo$ is a Bass order, the claim that $[I]=[(I:I)]$ follows from \cite[Remark 3.2.(1)]{CHL}.
 \end{proof}

\begin{remark}\label{charbassrmk}
\cite[Proposition 3.6]{CHL} describes a complete characterization for when $\mfo[X]/(\chi_\gamma(X))$ is a Bass order in a field $F[X]/(\chi_\gamma(X))$. We will summarize it in this remark.

Suppose that $\mfo[X]/(\chi_\gamma(X))$ is not the maximal order of $F[X]/(\chi_\gamma(X))$. Choose $K$ to be the maximal unramified extension of $F$ contained $F[X]/(\chi_\gamma(X))$ whose residue field is the same as that of $\mfo[X]/(\chi_\gamma(X))$.
Then choose an irreducible polynomial $g_\chi(X)\in \mfo_K[X]$ such that $\mfo[X]/(\chi_\gamma(X))\cong \mfo_K[X]/(g_\gamma(X))$.
Note that $\mfo_K[X]/(g_\gamma(X))\cong \mfo_K[X]/(g_\gamma(X-a))$ for $a\in \mfo_K$.

Then \cite[Proposition 3.6]{CHL} explains that up to translation by $X\mapsto X-a$,  $\mfo_K[X]/(g_\gamma(X))$ is a Bass order if and only if either $\ord_K(g_\gamma(0))=2$ or $[F[X]/(\chi_\gamma(X)):K]=2$.

If $k=\bar{k}$, then $\mfo_K=\mfo$ and thus $g_\gamma(X)=\chi_\gamma(X)$. 
Consequently $\mfo[X]/(\chi_\gamma(X))$  is a Bass order if and only if  either $n=2$ or $n$ is odd with $d=2$ up to translation, since $\chi_\gamma(X)$ is assumed to be irreducible over $\mfo$.
A Bass order does not have to be a simple extension of $\mfo$ (cf. Proposition \ref{thm:overorders}). 
\end{remark}

\section{Global geometrization of local smoothening process}\label{section3}
In this section, we geometrize the local smoothening method explained in Section \ref{sec:localsmoothening}, in the context of the global Hitchin fibration for $\mathrm{GL}_n$.
The primary objective is to construct a smooth stratification of the Hitchin fiber $\Phi^{-1}(\chi)$ in the Bass case (cf. Theorem  \ref{thm:smoothness} and Corollary \ref{cor:smoothstr} for $n\geq 3$, and Theorem \ref{thm:smoon=2} for $n=2$).
We first introduce the notations in the global setting, following \cite[Section 3]{Ch14}. 
\begin{itemize}
    \item Let $k$ be a finite field extension of $\mathbb{F}_p$ or an algebraic closure of $\mathbb{F}_p$. Suppose that $p>n$.
    \item Let $C$ be a smooth, geometrically connected, and projective curve defined over $k$.
    \item Unless otherwise specified, by a point $x\in C$ we mean $x\in C(k)$.
    \item For a $k$-scheme $S$, $C\times_{\mathrm{Spec}~k}S$ is denoted by $C_S$, and $C\times_{\mathrm{Spec}~k}\mathrm{Spec}~R$ by $C\times_kR$.
    \item Let $F$ be the function field of $C$. Here $F$ is different from the local setting in Section \ref{sec:localsmoothening}.
    \item For $x \in C$, let $\Mfo_x$ be the completed local ring of $C$ at $x$, and let $F_x$ be the fraction field of $\Mfo_x$.
    Choose a uniformizer $\pi_x$ of $\Mfo_x$, contained in $F$.
    \item Let $\mathcal{O}$ be the product of all completions $\prod\limits_{x\in C}\mathcal{O}_x$.
    \item Let $\mathbb{A}_C$ be the ring of adeles of $C$, the restricted product $\prod'_{x\in C}F_x$ with respect to $\mathcal{O}$. 
    \item $D=\sum\limits_{x} d_x[x]$ is an even and effective divisor on $C$ such that $\mathrm{deg}(D)>2g_C$, where $g_C$ is the genus of $C_{\bar{k}}$.
     Let $\pi_D=\prod\limits_{x\in C}\pi_x^{d_x}$, which is an element of $F$.
\end{itemize}
    
    \subsection{Construction of Hitchin fibration}\label{subsec:constofM}
    In this subsection, we will construct the elliptic Hitchin fibration, which is a proper morphism of schemes over $k$. 
 This is well explained in \cite[Sections 4.6-4.7]{CL10}.
 Since the construction is essential to our purpose, we will summarize it in this subsection. 

For a $k$-algebra $R$, a \textit{Hitchin bundle} over $R$ is a pair $(\mathcal{E}, \theta)$ where 
\begin{itemize}
    \item[-] $\mathcal{E}$ is a vector bundle on $C_R$ of rank $n$ and of degree $0$;
    \item[-] $\theta: \mathcal{E}\rightarrow \ce(D)=\mathcal{E}\otimes_{\mathcal{O}_{C_R}}\mathcal{O}_{C_R}(D)$ is a homomorphism of $\co_{C_R}$-modules.
\end{itemize}

Let $\mathbb{M}$ be the algebraic $k$-stack which classifies Hitchin bundles $(\ce, \theta)$ and let $\mathbb{A}:=\bigoplus\limits_{i=1}^nH^0(C,\mathcal{O}_C(iD))$ be an affine space over $k$.
Then the Hitchin fibration is the morphism of $k$-stacks
\[
\Phi:\mathbb{M} \longrightarrow \mathbb{A}, 
~~~~~~~   (\ce, \theta) \mapsto \chi_{\theta}=\left(a_1, \dots, a_n\right) \left(=\textit{coefficients of }\chi_{\theta}(X)\right), 
\]
where the characteristic polynomial $\chi_\theta(X)$ of $\theta$ is $X^n-a_1X^{n-1}+\dots +(-1)^na_n$ with $a_i\in H^0(C, \co_C(iD))$.
By abuse of notation, we often say that $\chi_\theta(X)$ is an element of $\mathbb{A}$.

Let $\mathbb{A}^{ell}$ be the open subscheme of $\mathbb{A}$ which consists of irreducible polynomials $\chi$ in $F[X]$, equivalently consists of characteristic polynomials of elliptic (and thus regular and semisimple) elements. 
Let $\mathbb{M}^{ell}=\mathbb{M}\times_{\mathbb{A}}\mathbb{A}^{ell}$, which is an open substack of $\mathbb{M}$. 

Assuming that such point exists, we choose $\infty\in C(k)$, which does not belong to the support of the divisor $D$ (so that $d_{\infty}=0$). 
Let $\mathbb{A}^{\infty}$ be the open subscheme of $\mathbb{A}^{ell}$ consisting of characteristic polynomials whose reduction modulo $\pi_{\infty}$, denoted by $\overline{\chi_{\infty}}$, has $n$-distinct simple roots in $k[X]$ (cf. \cite[Sections 3.4 and 11.1]{CL10}). 

Let $\mathcal{A}^{ell}$ be the \'etale Galois cover of $\mathbb{A}^{\infty}$ of the symmetric group $S_n$ given by 
\[
\mathcal{A}^{ell}(R)=\{(\chi, \tau)\in \mathbb{A}^{\infty}(R)\times R^n| \overline{\chi_{\infty}}=\prod_{i=1}^{n}(X-\tau_i)\in R[X]\}. 
\]
Then the fiber product $\mathbb{M}^{\infty}\times_{\mathbb{A}^{\infty}}\mathcal{A}^{ell}$ classifies quadruples $(\ce, \theta, \chi,\tau)$, where $(\ce, \theta)$ is a Hitchin bundle and $(\chi, \tau)\in \mathcal{A}^{ell}$ is such that $\chi=\Phi(\ce, \theta)$.  
Here $\mathbb{M}^{\infty}=\mathbb{M}\times_{\mathbb{A}}\mathbb{A}^{\infty}$ is an open substack of $\mathbb{M}$. 

Note that $\theta_{\infty}$ has $n$-distinct eigenvalues and $n$-linearly independent eigenvectors.
Let
\[
\mathcal{M}^{ell}\longrightarrow \mathbb{M}^{\infty}\times_{\mathbb{A}^{\infty}}\mathcal{A}^{ell}
\]
be the $\mathbb{G}_m$-torsor obtained by choosing an eigenvector $e_1$ corresponding to the eigenvalue $\tau_1$.
Finally we have a Hitchin morphism by base change
\[
\Phi^{ell}: \mathcal{M}^{ell}\longrightarrow \mathbb{M}^{\infty}\times_{\mathbb{A}^{\infty}}\mathcal{A}^{ell} \longrightarrow \mathcal{A}^{ell}. 
\]
Then \cite[Theorem 6.1.1]{Ch14} says that $\Phi^{ell}$ is proper and that $\mathcal{M}^{ell}$ is a smooth scheme over $k$.

The above discussion is summarized in the following diagram:
    
\begin{equation}\label{diag:1st}
           \begin{tikzcd}  
        \mathcal{M}^{ell} \arrow[rr, bend left, "\Phi^{ell}", "proper" swap] \arrow[r, "\textit{$\mathbb{G}_m$-torsor}"] \arrow[rd, "smooth" swap] \arrow[d, "smooth" swap]&
              [5em]   \mathbb{M}^{\infty}\times_{\mathbb{A}^{\infty}}\mathcal{A}^{ell} \arrow[r] \arrow[d, "\textit{\'etale}"] &  \mathcal{A}^{ell} \arrow[d, "\textit{\'etale}"] \ar[dl, phantom, "\square"]\\
  \mathrm{Spec}~k     & \mathbb{M}^{\infty} \arrow[r, "\Phi"] \arrow[d,hook,"open"] \ar[dr, phantom, "\square"] &  \mathbb{A}^{\infty} \arrow[d,hook, "open"]  \\
       & \mathbb{M}^{ell} \arrow[r, "\Phi"]    \arrow[d,hook,"open"]             &  \mathbb{A}^{ell} \arrow[d,hook,"open"] \ar[dl, phantom, "\square"]\\
       & \mathbb{M} \arrow[r, "\Phi"]                &  \mathbb{A}=\bigoplus\limits_{i=1}^nH^0(C,\mathcal{O}_C(iD)).
        \end{tikzcd}
    \end{equation}

    Here $\square$ means a Cartesian diagram.
We will use $\Phi$ for $\Phi^{ell}$ from Section \ref{subsec:globalsm} if there is no confusion.

\begin{remark}\label{rmk:adelic}
\begin{enumerate}
\item{[\textbf{Adelic description}]}  
In this remark, we give an adelic description of $\mathbb{M}(R)$ for  a $k$-algebra $R$, following 
\cite[Section 7.3]{CL10}. Here loc. cit. assumes $k=\bar{k}$, but the argument also holds over $k$  by \cite{BL}. 
We write an adelic group $\mathrm{GL}_n(\mathbb{A}_C\otimes_k R)=\prod'_{x\in C}\mathrm{GL}_n(F_x\otimes_k R)$ so that each element $g=(g_x)_{x\in C}\in \mathrm{GL}_n(\mathbb{A}_C\otimes_k R)$ means that $g_x\in \mathrm{GL}_n(\Mfo_x\otimes_k R)$ for all but finitely many points $x\in C$. We also write $\mathrm{GL}_n(\Mfo\otimes_k R)=\prod_{x\in C}\mathrm{GL}_n(\Mfo_x\otimes_k R)$.
Then $\mathrm{GL}_n(F\otimes_k R)$ acts on $\mathrm{GL}_n(\mathbb{A}_C\otimes_k R)$ by left multiplication, $\mathrm{GL}_n(\mathcal{O}\otimes_k R)$ acts on $\mathrm{GL}_n(\mathbb{A}_C\otimes_k R)$ by right multiplication, and $\mathrm{GL}_n(F\otimes_k R)$ acts on $\mathfrak{gl}_n(F\otimes_kR)$ by the adjoint action. 
By \cite[Corollary 7.3.2]{CL10},
$\mathbb{M}(R)$ consists of pairs $$(g,\theta)\big(=((g_x)_{x\in C}),\theta)\big)\in \mathrm{GL}_n(F\otimes_k R)\backslash \Big(\mathrm{GL}_n(\mathbb{A}_C\otimes_k R)/\mathrm{GL}_n(\mathcal{O}\otimes_k R)\times \mathfrak{gl}_n(F\otimes_kR)\Big) \textit{ such that}$$ 
    \begin{enumerate}
        \item The degree of $\det(g)$ is $0$. 
        \item The characteristic polynomial of $\theta$ lies in $ \mathbb{A}(R)$.
        \item $g^{-1}\theta g\in \pi_D^{-1}\gl_n(\mathcal{O}\otimes_kR)$, equivalently $g_x^{-1}\theta g_x\in \pi_x^{-d_x}\gl_n(\mathcal{O}_x\otimes_kR)$ for all $x\in C$.
    \end{enumerate}
Here $deg(\mathrm{det}(g))=\sum\limits_{x\in C}\ord(\mathrm{det}(g_x))$.
Then $\Phi((g,\theta))$ consists of the coefficients of $\chi_\theta(X)$.



\item{[\textbf{Translation invariance}]} This is taken from \cite[Lemma 3.4]{CKL}. 
For $c\in H^0(C, \Mfo_C(D))$ and for $\theta: \mathcal{E}\rightarrow \ce(D)$, let $\theta_c:=\theta+c\cdot id: \mathcal{E}\rightarrow \ce(D)$.  
We define $\iota_c:\mathbb{A}\rightarrow \mathbb{A}$ such that the coefficients of $\chi(X)$ maps to the coefficients of $\chi(X-c)$. Then the following diagram is commutative:
\[\xymatrixcolsep{5pc}\xymatrix{
\mathbb{M}  \ar[d]^{(\ce, \theta) \mapsto (\ce, \theta_c)} \ar[r]^{\Phi} &  \mathbb{A} \ar[d]^{\iota_c} \\ \mathbb{M}  \ar[r]^{\Phi} & \mathbb{A}.
}\]
Since $\iota_c\circ \iota_{-c}=\iota_{-c}\circ \iota_{c}=\iota_0$ is the identity on $\mathbb{A}$, $\iota_c$ is an automorphism. Thus $\Phi^{-1}(\chi(X))$ is isomorphic to $\Phi^{-1}(\chi(X-c))$.

\end{enumerate}

\end{remark}

 




    For $x\in C(k)$ and for $\chi(X)\in \mathbb{A}(k)$, 
we let
\begin{equation}\label{eq:chi_x}
    \chi_x(X):=\pi_x^{n\cdot d_x}\cdot\chi(X/\pi_x^{d_x}) \in \co_x[X] \textit{ and } \overline{\chi_x(X)}:=\chi_x(X) \textit{ modulo }\pi_x\in k[X].
\end{equation}

Note that 
$\chi_x(X)=X^n-\pi_x^{d_x}\cdot a_1X^{n-1}+\dots +(-1)^n\pi_x^{n\cdot d_x}\cdot a_n$.

\begin{definition}\label{def:mathcalS}
For $\chi(X)\in\mathbb{A}(k)$, we define  $\mathcal{S}$ to be the (finite) set of points in $C(k)$ such that the polynomial $\chi_x(X)$ as an element of $\Mfo_x[X]\otimes_k\bar{k}$ is irreducible.
\end{definition}

\begin{lemma}\label{lem:trans}
 Suppose that $\deg(D)\geq 2g_C-1+\#\mathcal{S}$.
  Then,  for $\chi(X)\in \mathbb{A}^{ell}(k)$, there exists $c\in H^0(C, \Mfo_C(D))$ such that the constant term of $\chi(X+c)$, denoted by $a_n^c$, satisfies $n\cdot d_{x}+\ord_{x}(a^c_n)>0$ for all $x$'s in $\mathcal{S}$.
\end{lemma}

\begin{proof}
Let $x\in\mathcal{S}$.
Then, by Hensel's Lemma, $\overline{\chi_x(X)}=(X-\overline{b}_x)^n$ for some $\overline{b}_x\in k$.
We prove the existence of $c\in H^0(C,\Mfo_C(D))$ whose reduction modulo the maximal ideal at each $x\in \mathcal{S}$ is equal to $\overline{b}_x$. 
The exact sequence $0\rightarrow \Mfo_C(D-S)\rightarrow \Mfo_C(D)\rightarrow \bigoplus\limits_{x\in\mathcal{S}}k\rightarrow 0$ induces a long exact sequence $$0\rightarrow H^0(C,\Mfo_C(D-S))\rightarrow H^0(C,\Mfo_C(D))\rightarrow \bigoplus\limits_{x\in\mathcal{S}}k\rightarrow H^1(C,\Mfo_C(D-S))\rightarrow \dots$$
 We claim that $H^1(C,\Mfo_C(D-S))=0$, so that the map $H^0(C,\Mfo_C(D))\rightarrow \bigoplus\limits_{x\in\mathcal{S}}k$ is surjective.
 By Serre duality, $H^1(C,\Mfo_C(D-S))\cong H^0(C,\Mfo(S-D+K))^\vee$, where $K$ is the canonical divisor.
 The degree of $S-D+K$ is equal to $\deg(S-D)+\deg(K)=\#\mathcal{S}-\deg(D)+(2g_C-2)$.
 Therefore, by the assumption on $\deg(D)$, we have $\deg(S-D+K)< 0$, so that $ H^0(C,\Mfo(S-D+K))$ vanishes.
 Choosing $c$ as an element in the preimage of $(\overline{b}_x)_{x\in \mathcal{S}}$ under the map $H^0(C,\Mfo_C(D))\rightarrow \bigoplus\limits_{x\in \mathcal{S}}k$ then concludes the proof. 
\end{proof}

\begin{setting}\label{settings}
\begin{enumerate}
    \item 
From now on until the end of the paper, we fix an irreducible polynomial $\chi$ over $F$ given by $(a_1, a_2, \dots, a_n)\in \mathbb{A}^\infty(k)$.
We also fix $(\chi, \tau)\in \mathcal{A}^{ell}(k)$ which lies over $\chi$.

\item 
We define the set of points $\mathcal{S}$ associated to $\chi$ following Definition \ref{def:mathcalS}.
We suppose that 
\begin{enumerate}
    \item the polynomial $\overline{\chi_x(X)}$ is separable over $k$ for every $x\not\in \mathcal{S}$. 
    \item $\deg(D)\geq 2g_C-1+\#\mathcal{S}$.
\end{enumerate}
Here by Remark \ref{rmk:adelic}.(2) the Hitchin fiber $\Phi^{-1}(\chi)$ is invariant under translation, up to isomorphism over $k$. 
Thus by Lemma \ref{lem:trans}, we may and do assume that  
\begin{equation}
    n\cdot d_{x}+\ord_{x}(a_n)>0  ~~~~~   \textit{  if and only if  }  ~~~~~  x\in \mathcal{S}.
    \end{equation}
\end{enumerate}
  \end{setting}

\subsection{Construction of $\mathcal{M}^{\mathcal{F}}$}\label{subsec:mf}

We first recall a stack of coherent sheaves on $C$, denoted by $Coh(C)$, following \cite[\href{https://stacks.math.columbia.edu/tag/08KA}{Tag 08KA}]{stacks-project}. 
For a $k$-scheme $S$,  $Coh(C)(S)$ is the category of quasi-coherent $\Mfo_{C_S}$-sheaves of finite presentation, flat, and of proper support over $S$.
Then $Coh(C)$ is an algebraic stack over $k$ by \cite[\href{https://stacks.math.columbia.edu/tag/08WC}{Tag 08WC}]{stacks-project}. 
Let $Coh(C)_0$ be the substack of $Coh(C)$ parametrizing sheaves of $0$-dimensional support. 
Then $Coh(C)_0$ is an open substack of $Coh(C)$, which is explained in \cite[Section 3.1]{FR}.

\begin{lemma}\label{lem:repre}
For a $k$-scheme $S$, define $\mathcal{Q}_S$ as follows:
\[
\mathcal{Q}_S: \mathcal{M}^{ell}(S) \longrightarrow \mathbb{M}^{ell}(S) \longrightarrow Coh(C)_0(S), ~~~~~~ (\mathcal{E}, \theta) \in \mathbb{M}^{ell}(S)  \mapsto  \textit{cokernel of }\theta.
\]
Then it forms a morphism of $k$-stacks, denoted by $\mathcal{Q}$.
\end{lemma}

\begin{proof}

For a Hitchin bundle $(\mathcal{E},\theta)\in \mathbb{M}^{ell}(S)$, since $\chi_\theta$ is irreducible, $\theta$ is an isomorphism outside a finite set of points, and thus its cokernel has relative dimension 0 over $S$.

It remains to show that $\mathcal{Q}$ commutes with base change. 
For a $k$-scheme morphism $S'\rightarrow S$, let $(\mathcal{E}',\theta')\in \mathcal{M}(S')$ be the image of $(\theta, \mathcal{E})\in \mathcal{M}(S)$, which is defined to be the base change of $(\mathcal{E},\theta)$ to $\mathcal{O}_{S'}$. 
Then it suffices to show that the cokernel of $\theta'$ is isomorphic to the base change of the cokernel of $\theta$ to $\mathcal{O}_{S'}$. This follows from the right exactness of the tensor product.    
\end{proof}

\begin{definition}\label{def:coherentsheafF}
We say that \textit{a coherent sheaf $\mathcal{F}$ on $C$ is associated to $\chi$} if the support of $\mathcal{F}$ is  $\mathcal{S}$ such that the length of $\mathcal{F}_x$ with $x\in \mathcal{S}$, as an $\Mfo_x$-module, is $n\cdot d_x+\ord_x(a_n) \left(>0\right)$. 
\end{definition}

Note that such $\mathcal{F}$ is a direct sum of skyscraper sheaves supported at $\mathcal{S}$ and thus exists, since the stalk functor is adjoint to the skyscraper functor (cf. \cite[\href{https://stacks.math.columbia.edu/tag/009C}{Lemma 009C}]{stacks-project}). Loc. cit. also yields that if two coherent sheaves with finite support are isomorphic at each stalk then they are isomorphic.

Then $\mathcal{F}$ defines a unique and locally closed point in $Coh(C)_0(k)$ by 
\cite[Theorem 3.5]{FR}.
Loc. cit. yields that $\mathcal{F}$ is a closed point of $Coh(C)_0(k)$ if and only if each stalk $\mathcal{F}_x$ with $x\in \mathcal{S}$ is isomorphic to $k$. 
Note that this can also be seen from the stratification of $Coh(C)_0$ given in \cite[Section 2.2.2]{Hen24}.

\begin{definition}\label{def:mf}
For a coherent sheaf $\mathcal{F}$ in Definition \ref{def:coherentsheafF}, define 
$\mathcal{M}^{\mathcal{F}}:=\mathcal{M}^{ell}\times_{Coh(C)_0}\mathrm{Spec}~k$, where $\mathrm{Spec}~k\rightarrow Coh(C)_0$ gives the sheaf $\mathcal{F}$.
Similarly, we define $\mathbb{M}^{\mathcal{F}}:=\mathbb{M}^{ell}\times_{Coh(C)_0}\mathrm{Spec}~k$.
\end{definition}

\begin{remark}\label{rmk:locallyclosed}
 $\mathcal{M}^{\mathcal{F}}$ is a locally closed substack of $\mathcal{M}^{ell}$ by Lemma \ref{lem:repre} and \cite[\href{https://stacks.math.columbia.edu/tag/0501}{Lemma 0501}]{stacks-project}, and thus is an algebraic space by \cite[\href{https://stacks.math.columbia.edu/tag/0503}{Lemma 0503}]{stacks-project}.
\cite[Corollary II.6.16]{Knu71} then yields that it is representable by a scheme.
In conclusion, 
$\mathcal{M}^{\mathcal{F}}$ is a locally closed subscheme of $\mathcal{M}^{ell}$, and it is closed if and only if $\mathcal{F}_x\cong k$ for any $x\in \mathcal{S}$. 
\end{remark}

\begin{remark}\label{rmk:ademf}
We give an adelic description of $\mathbb{M}^{\mathcal{F}}(R)$, where $R$ is a $k$-algebra, using Remark \ref{rmk:adelic}.
Consider $(g,\theta)\big(=((g_x)_{x\in C}),\theta)\big)\in \mathbb{M}^{ell}(R)$ so that $g_x^{-1}\theta g_x\in \pi_x^{-d_x}\gl_n(\mathcal{O}_x\otimes_kR)$ for all $x\in C$.
Then $(g,\theta)\big(=((g_x)_{x\in C}),\theta)\big)\in \mathbb{M}^{\mathcal{F}}(R)$ if and only if 
the cokernel of $g_x^{-1}\theta g_x$ is isomorphic to $\mathcal{F}_x\otimes_kR$ for all $x\in C$. 
Here by the cokernel of $g_x^{-1}\theta g_x$, we view it as a morphism from $(\mathcal{O}_x\otimes_kR)^n$ to $(\pi_x^{-d_x}\mathcal{O}_x\otimes_kR)^n$.
Note that the cokernel is non-trivial if and only if $x\in \mathcal{S}$. 
We emphasize that the cokernel of $g_x^{-1}\theta g_x$ with $x\in \mathcal{S}$ exactly defines the type in the sense of Definition \ref{def_type}.
That is, 
\[\textit{Cokernel of }g_x^{-1}\theta g_x\cong \mathcal{F}_x\otimes_kR\cong \left(\Mfo_x/(\pi_x^{k_1})\oplus\Mfo_x/(\pi_x^{k_2})\oplus \dots \oplus \Mfo_x/(\pi_x^{k_n})\right)\otimes_kR\] 
for $x\in \mathcal{S}$, where $k_1+\dots+k_n=n\cdot d_x+\ord_x(a_n)>0$.

\end{remark}




\subsection{Construction of $\mathcal{M}^{\mathcal{F}, l_x}$}\label{sec:bassglobal}

Let $n\geq 3$. 
Let $\mathcal{F}$ be a coherent sheaf such that $\mathcal{F}_x\cong k\oplus k$ for a certain $x\in \mathcal{S}$ in Definition \ref{def:coherentsheafF}. 
We consider an Artin stack $[\mathfrak{gl}_n/\mathrm{GL}_n]$ over $k$, the quotient of the affine space $\mathfrak{gl}_n\cong \mathbb{A}_k^{n^2}$ by the $\mathrm{GL}_n$-action given by conjugation.
For a $k$-algebra $R$, a point in $[\mathfrak{gl}_n/\mathrm{GL}_n](R)$ is the groupoid of an orbit of the action of $\mathrm{GL}_n(R)$ on $\mathfrak{gl}_n(R)$.

\begin{lemma}
    For a $k$-algebra $R$ and for $x\in \mathcal{S}$, define $\mathcal{R}^x_R$ as follows:
    \[
\mathcal{R}^x_R: \mathcal{M}^{\mathcal{F}}(R) \longrightarrow \mathbb{M}^{\mathcal{F}}(R) \longrightarrow [\mathfrak{gl}_n/\mathrm{GL}_n](R), ~~~~~~ (g,\theta)\in (\mathbb{M}^{\mathcal{F}})(R)\mapsto \overline{\pi_x^{d_x}g_x^{-1}\theta g_x}.
\]
Here $(g, \theta)$ is an adelic description of a point in $\mathbb{M}^{\mathcal{F}}(R)$ (cf. Remark \ref{rmk:ademf})
and $\overline{\pi_x^{d_x}g_x^{-1}\theta g_x}$ is the reduction of $\pi_x^{d_x}g_x^{-1}\theta g_x$ modulo $\pi_x$.
Then it forms a morphism of $k$-stacks, denoted by $\mathcal{R}^x$.
\end{lemma}

\begin{proof}
Since schemes are locally affine, it suffices to check that $\mathcal{R}^x$ yields a well-defined natural transformation between contravariant functors from the category of $k$-algebras to the category of groupoids.
If $(g,\theta)$ and $(g',\theta')$ are two representatives of a same point in $\mathbb{M}^{\mathcal{F}}(R)$, then $\pi_x^{d_x}g_x^{-1}\theta g_x$ and $\pi_x^{d_x}{g'_x}^{-1}\theta'g_x'$ are conjugate by an element of $\mathrm{GL}_n(\mathcal{O}_x\otimes_kR)$,
and thus their reductions modulo $\pi_x$ are conjugate by an element of $\mathrm{GL}_n(R)$.
This yields the well-definedness of $\mathcal{R}^x$.
The naturality of $\mathcal{R}^x$ follows from the right exactness of the tensor product, as in the proof of Lemma \ref{lem:repre}.
Therefore, $\mathcal{R}^x$ gives a well-defined morphism of Artin stacks.
\end{proof}

Let  $J_{i}$ be the Jordan canonical form of size $i$ with diagonal entries $0$ and with superdiagonal entries $1$, and let $J_{i,j}:=J_i\perp J_j$. Then $J_{i,j} \in \gl_n(k)$ with $i+j=n$ represents a point in $[\mathfrak{gl}_n/\mathrm{GL}_n](k)$.

\begin{definition}\label{def:constmfl}
For an integer $0\leq l_x\leq n/2-1$ with a fixed $x\in \mathcal{S}$, we define $\mathcal{M}^{\mathcal{F},l_x}$ to be $(\mathcal{R}^x)^{-1}(J_{l_x+1, n-l_x-1})$. More precisely, $\mathcal{M}^{\mathcal{F},l_x}:=\mathcal{M}^{\mathcal{F}}\times_{[\mathfrak{gl}_n/\mathrm{GL}_n]}\mathrm{Spec}~k$, where $\mathrm{Spec}~k \rightarrow [\mathfrak{gl}_n/\mathrm{GL}_n]$ assigns $J_{l_x+1, n-l_x-1}$. 
Then $\mathcal{M}^{\mathcal{F}, l_x}$ is a locally closed subscheme of $\mathcal{M}^{\mathcal{F}}$ as in Remark \ref{rmk:locallyclosed}.
\end{definition}


\begin{remark}\label{rmk:ademfly}
We give an adelic description of $\mathbb{M}^{\mathcal{F},l_x}(R)$ with a $k$-algebra $R$, using Remark \ref{rmk:ademf}.

Consider $(g,\theta)\big(=((g_x)_{x\in C}),\theta)\big)\in \mathbb{M}^{ell}(R)$ so that $g_x^{-1}\theta g_x\in \pi_x^{-d_x}\gl_n(\mathcal{O}_x\otimes_kR)$ for all $x\in C$. 
Then $(g,\theta)\in \mathbb{M}^{\mathcal{F},l_x}(R)$ if and only if 
 $(g,\theta)\in \mathbb{M}^{\mathcal{F}}(R)$ (cf. Remark \ref{rmk:ademf}) and 
 $\overline{\pi_x^{d_x}g_x^{-1}\theta g_x} \left(\in \gl_n(R)\right)$ is conjugate to $J_{l_x+1, n-l_x-1}$ by an element of $\mathrm{GL}_n(R)$, for a fixed $x\in \mathcal{S}$. 
\end{remark}

\begin{lemma}\label{lem:fatpoints}
Choose $(g,\theta)\in \mathbb{M}^{\mathcal{F}}(R)$ for $R=k[\epsilon]/(\epsilon^m)$, $m\geq 1$. 
To reduce the notation, we let 
\[
\left\{
  \begin{array}{l}
M_x\left(\in \gl_n(\mathcal{O}_x\otimes_kR)\right):=\pi_x^{d_x}g_x^{-1}\theta g_x \textit{ so that } \overline{M_x}\left(\in \gl_n(R)\right):=\pi_x^{d_x}g_x^{-1}\theta g_x \textit{  modulo $\pi_x$};\\
\textit{$M_{x,0}\left(\in \gl_n(\mathcal{O}_x)\right):=M_x$ modulo $\epsilon$ so that $\overline{M_{x,0}}\left(\in \gl_n(k)\right):=M_{x,0}$ modulo $\pi_x$}.
    \end{array} \right.
\]    

Then $(g,\theta)\in \mathbb{M}^{\mathcal{F},l_x}(R)$ if and only if $M_x$ is conjugate to $\begin{pmatrix}
J_{l_x, n-2-l_x}+\pi X^\dag_{1,1} &X_{1,2} & X_{1,3} \\
\pi X_{2,1} &\pi X_{2,2} & \pi X_{2,3} \\
\pi X_{3,1} &\pi X_{3,2} & \pi X_{3,3}
\end{pmatrix}\in \mathrm{GL}_n(\Mfo_x\otimes_kR)$ by an element of $\mathrm{GL}_n(\Mfo_x\otimes_kR)$.
Here $J_{l_x, n-2-l_x}:=J_{l_x}\perp J_{n-2-l_x}\in \gl_{n-2}(k)$. 
\end{lemma}
\begin{proof}
A key ingredient in the proof is \cite[Theorem III.2]{Mc78}, which states the Jordan decomposition theorem over an Artinian local principal ideal ring $R$.
In our case, it is stated as follows. If there exists an integer $s> 0$ such that a square matrix $A$ over $R$ satisfies $A^s=0$ in $\mathfrak{gl}_n(R)$ and the reduction of $A^{s-1}$ modulo $\epsilon$ is nonzero in $\gl_n(k)$, then $A$ is conjugate to a Jordan block matrix with zero diagonal, by an element of $\mathrm{GL}_n(R)$.
Here, a Jordan block matrix with zero diagonal is a block diagonal matrix whose diagonal blocks have $1$ on its superdiagonal entries and $0$ elsewhere.


\begin{enumerate}
    \item 
    We first prove the `if' part. 
Suppose that $M_x$ is conjugate to $\begin{pmatrix}
J_{l_x, n-2-l_x}+\pi_x X^\dag_{1,1} &X_{1,2} & X_{1,3} \\
\pi_x X_{2,1} &\pi_x X_{2,2} & \pi_x X_{2,3} \\
\pi_x X_{3,1} &\pi_x X_{3,2} & \pi_x X_{3,3}
\end{pmatrix}\in \gl_n(\Mfo_x\otimes_kR)$ by an element of $\mathrm{GL}_n(\Mfo_x\otimes_kR)$ such that $\begin{pmatrix}
J_{l_x, n-2-l_x}+\pi_x X^\dag_{1,1} &X_{1,2} & X_{1,3} \\
 X_{2,1} & X_{2,2} &  X_{2,3} \\
 X_{3,1} & X_{3,2} &  X_{3,3}
\end{pmatrix}$ is invertible. 
Then we claim that $(g,\theta)\in \mathbb{M}^{\mathcal{F},l_x}(R)$. 

Since $M_x$ is defined up to conjugation by an element of $\mathrm{GL}_n(\Mfo_x\otimes_k R)$, 
we may and do write $M_x=\begin{pmatrix}
J_{l_x, n-2-l_x}+\pi_x X^\dag_{1,1} &X_{1,2} & X_{1,3} \\
\pi_x X_{2,1} &\pi_x X_{2,2} & \pi_x X_{2,3} \\
\pi_x X_{3,1} &\pi_x X_{3,2} & \pi_x X_{3,3}
\end{pmatrix}$ satisfying the invertibility condition stated in the claim.

Thus, 
it suffices to show that $\overline{M_x}$ is conjugate to $J_{l_x+1, n-l_x-1}$ by an element of $\mathrm{GL}_n(R)$, by Remark \ref{rmk:ademfly}.
We put $l_x':=n-2-l_x$ so that $l_x'\geq l_x$ to simplify notation. 
To use the Jordan decomposition theorem of \cite[Theorem III.2]{Mc78} stated at the beginning of the proof, we claim that $\overline{M_x}^{l_x'+1}=0$, but $\overline{M_{x,0}}^{l_x'}\neq 0$. 
For the first claim, the fact that $J_{l_x, l_x'}^{l_x'}=0$ yields 
\[
\overline{M_x}^{l_x'+1}=\begin{pmatrix}
J_{l_x, l_x'}^{l_x'}\cdot J_{l_x, l_x'} &J_{l_x, l_x'}^{l_x'}\cdot \overline{X_{1,2}} & J_{l_x, l_x'}^{l_x'}\cdot \overline{X_{1,3}} \\
0&0&0\\
0&0&0
\end{pmatrix}=0.
\]

For the second claim, we observe that $\overline{M_{x,0}}$ is of the form $\begin{pmatrix}
J_{l_x, l_x'} &\ast & \ast\ast \\
0&0&0\\ 0&0&0
\end{pmatrix}\in \gl_n(k)$ which is of rank $n-2$. 
This directly implies that $\overline{M_{x,0}}$ is conjugate to $J_{l_x+1, l_x'+1}$ by an element of $\mathrm{GL}_n(k)$. 
In particular, the minimal polynomial of $\overline{M_{x,0}}$ is $X^{l_x'+1}$. 
This yields that $\overline{M_{x,0}}^{l_x'}\neq 0\in  \mathfrak{gl}_n(k)$, which verifies the second claim. 

To sum up, $\overline{M_x}$ is conjugate to a Jordan block matrix with zero diagonal (which we denote by $J$), by an element of $\mathrm{GL}_n(R)$. 
Then $\overline{M_{x,0}}$ is also conjugate to $J$, by an element of $\mathrm{GL}_n(k)$ and thus $J=J_{l_x+1, n-l_x-1}$. This completes the `if' part.

    \item For the `only if' part, suppose that $(g,\theta)\in \mathbb{M}^{\mathcal{F},l_x}(R)$.
Since $M_x$ is defined up to conjugation by an element of $\mathrm{GL}_n(\Mfo_x\otimes_k R)$, 
we may and do write $M_x=\begin{pmatrix}
X_{1,1} &X_{1,2} & X_{1,3} \\
\pi_x X_{2,1} &\pi_x X_{2,2} & \pi_x X_{2,3} \\
\pi_x X_{3,1} &\pi_x X_{3,2} & \pi_x X_{3,3}
\end{pmatrix}\in \gl_n(\Mfo_x\otimes_kR)$ such that $\begin{pmatrix}
X_{1,1} &X_{1,2} & X_{1,3} \\
 X_{2,1} & X_{2,2} &  X_{2,3} \\
 X_{3,1} & X_{3,2} &  X_{3,3}
\end{pmatrix}$ is invertible, where the size of $X_{1,1}$ is $(n-2)\times (n-2)$, and so on (cf. Remark \ref{rmk:ademf}), and such that $\overline{M_x}$ is conjugate to $J_{l_x+1, n-l_x-1}$ by an element of $\mathrm{GL}_n(R)$ (cf. Remark \ref{rmk:ademfly}).

We first claim that $\overline{X_{1,1}}$ is conjugate to $J_{l_x, n-2-l_x}$ by an element of $\mathrm{GL}_{n-2}(R)$. 
We again put $l_x':=n-2-l_x$ so that $l_x'\geq l_x$ to simplify notation. 
To use the Jordan decomposition theorem of \cite[Theorem III.2]{Mc78} stated at the beginning of the proof, we claim that $\overline{X_{1,1}}^{l_x'}=0$ in $\gl_{n-2}(R)$ and that the reduction of $\overline{X_{1,1}}^{l_x'-1}$ modulo $\epsilon$ is nonzero in $\gl_{n-2}(k)$.

Since $\overline{M_x}$ is conjugate to $J_{l_x+1, l_x'+1}$ by an element of $\mathrm{GL}_n(R)$, we have $\overline{M_{x}}^{l_x'+1}=0$ in $\gl_n(R)$ and $\overline{M_{x,0}}^{l_x'+1}=0$ in $\gl_n(k)$.
This directly shows that the minimal polynomial of the reduction of $\overline{X_{1,1}}$ modulo $(\epsilon)$ is $X^{l_x'}$ in $\gl_{n-2}(k)$ because the two lowest rows of $\overline{M_{x,0}}$ are zero.
In particular, the reduction of $\overline{X_{1,1}}^{l_x'-1}$ modulo $\epsilon$ is nonzero in $\gl_{n-2}(k)$.
This proves the second claim.
For the first claim, we consider the following matrix equation:
\[
0=\overline{M_x}^{l_x'+1}=\begin{pmatrix}
\overline{X_{1,1}}^{l_x'}\cdot \overline{X_{1,1}} &\overline{X_{1,1}}^{l_x'}\cdot \overline{X_{1,2}} & \overline{X_{1,1}}^{l_x'}\cdot \overline{X_{1,3}} \\
0&0&0\\
0&0&0
\end{pmatrix}\in \gl_n(R).
\]
Suppose that $\overline{X_{1,1}}^{l_x'}\neq 0$ in $\gl_{n-2}(R)$.
Then all of column vectors of $\overline{X_{1,1}}, \overline{X_{1,2}}, \overline{X_{1,3}}$ are contained in the kernel of $\overline{X_{1,1}}^{l_x'}$. 
On the other hand, these column vectors span $R^{n-2}$ since  $\begin{pmatrix}
\overline{X_{1,1}} &\overline{X_{1,2}} & \overline{X_{1,3}} \\
 \overline{X_{2,1}} & \overline{X_{2,2}} &  \overline{X_{2,3}} \\
 \overline{X_{3,1}} & \overline{X_{3,2}} &  \overline{X_{3,3}}
\end{pmatrix}\in \gl_n(R)$ is invertible.
This contradicts the assumption that $\overline{X_{1,1}}^{l_x'}\neq 0$.

Therefore, the Jordan decomposition theorem yields that $\overline{X_{1,1}}$ is conjugate to a Jordan block matrix by an element of $\mathrm{GL}_{n-2}(R)$. 
By the invertibility condition, this Jordan block matrix can have at most two zero columns, and hence at most two diagonal blocks.
By this condition and the equalities $\overline{X_{1,1}}^{l_x'}=0$ and $\overline{X_{1,1}}^{l_x'-1}\neq 0$, this Jordan block matrix is exactly $J_{l_x,l_x'}$, which verifies the desired claim.

Choose $h_{n-2}\in \mathrm{GL}_{n-2}(R)$ such that the conjugation of $\overline{X_{1,1}}$ by $h_{n-2}$ is $J_{l_x,l_x'}$ in $\gl_{n-2}(R)$.
Then $h_n:=\begin{pmatrix}
    h_{n-2} & 0 \\ 0 & id_2
\end{pmatrix}$ is an element of $\mathrm{GL}_n(R)$, which is contained in $\mathrm{GL}_n(\Mfo_x\otimes_kR)$.
In conclusion, the conjugation of $M_x$ by $h_n$ yields the desired claim. 
\end{enumerate}
\end{proof}

\subsection{Algebraic stratification}\label{subsec:globalsm}    
  In this subsection, we define a global type of $\chi$ and then define the restricted Hitchin fibration. 
  To simplify notation, we use $\Phi: \mathcal{M}^{ell}\longrightarrow  \mathcal{A}^{ell}$ for $\Phi^{ell}: \mathcal{M}^{ell}\longrightarrow  \mathcal{A}^{ell}$.  
    We will  treat the case for $n\geq 3$ and the case for $n=2$ separately.

\subsubsection{Algebraic stratification  with $n\geq 3$}\label{subsubsecn>3}
We suppose that $n\geq 3$ in this subsection.

    \begin{definition}\label{def:typeofchia}
   The \textit{type} of $\chi$  is a tuple $T=(\mathcal{F}, (l_x)_{x\in \mathcal{S}^{\mathcal{F}}_1})$ where:
        \begin{enumerate}
            \item $\mathcal{F}$: as in Definition \ref{def:coherentsheafF} such that $\mathcal{F}_x\cong k\oplus k$ or $\mathcal{F}_x\cong \Mfo_x/(\pi_x^{n\cdot d_x+\ord_x(a_n)})$ for $x\in \mathcal{S}$;
            \item  $\mathcal{S}^{\mathcal{F}}_1:=\{x\in \mathcal{S}|\mathcal{F}_x\cong k\oplus k\}$, and  $l_x$: an integer with  $0\leq l_x \leq n/2-1$.
        \end{enumerate}
    \end{definition}
    
Then, for a type $T=(\mathcal{F},(l_x)_{x\in \mathcal{S}^{\mathcal{F}}_1})$ of $\chi$, repeating the construction of Section \ref{sec:bassglobal} for each $x\in \mathcal{S}^{\mathcal{F}}_1$ yields a locally closed subscheme 
$\mathcal{M}^T$ (resp. $\mathbb{M}^T$) of $\mathcal{M}$ (resp. $\mathbb{M}^{ell}$).
We define a closed subscheme $\mathbb{A}^{T}$ of $\mathbb{A}^\infty$ such that
$$\mathbb{A}^{T}:=
   \left(\bigoplus\limits_{i=1}^{n} H^0\Big(C,\mathcal{O}_C\big(iD-\sum\limits_{x\in\mathcal{S}^{\mathcal{F}}_1} a_{i,x}[x]-\sum\limits_{x\in\mathcal{S}\setminus\mathcal{S}^{\mathcal{F}}_1} b_{i,x}[x]\big)\Big)\right)\cap \mathbb{A}^\infty.
   $$
   Here $  a_{i,x}=\left\{ \begin{array}{l l}
1   & \textit{if $i\leq n-l_x-1$};\\
2 & \textit{if $i> n-l_x-1$}
    \end{array} \right.$ and $  b_{i,x}=\left\{ \begin{array}{l l}
 0  & \textit{if $i\leq n-1$};\\
n\cdot d_{x}+\ord_{x}(a_n) & \textit{if $i=n$}
    \end{array} \right.$

Let $\mathcal{A}^{T}:=\mathcal{A}^{ell}\times_{\mathbb{A}^\infty} \mathbb{A}^{T}$.
Then we have the following result about the restricted Hitchin fibration.
\begin{lemma}\label{lem:rest}
For each type $T$ of $\chi$,
the image of  $\Phi$ restricted to $\mathcal{M}^{T}$ lies in 
$ \mathcal{A}^{T}$.
\end{lemma}
\begin{proof}
For a $k$-algebra $R$, let $(g,\theta)\in \mathbb{M}^{T}(R)$, following Remark \ref{rmk:adelic}.
It suffices to show that the tuple of the coefficients of the characteristic polynomial of $\theta$ is contained in $\mathbb{A}^{T}(R)$. 

Firstly, consider  $x\in \mathcal{S}^{\mathcal{F}}_1$.
By Remark \ref{rmk:ademfly}, an element $\pi_{x}^{d}g_{x}^{-1}\theta g_{x} \in \gl_n(\Mfo_x\otimes_kR)$ is conjugate to $J_{l_x+1,n-l_x-1}+\pi_{x} M\in \mathfrak{gl}_n(\mathcal{O}_{x}\otimes_k R)$ for a certain $M\in \mathfrak{gl}_n(\mathcal{O}_{x}\otimes_k R)$.
The coefficients of the characteristic polynomial of $J_{l_x+1,n-l_x-1}+\pi_{x} M$ are then contained in $(\pi_{x}\Mfo_{x}\otimes_k R)^{n-l_x-1}\times (\pi_{x}^2\Mfo_{x}\otimes_k R)^{l_x+1}$.
Since the characteristic polynomial of $\theta$ equals that of $g_{x}^{-1}\theta g_{x}$, 
it is contained in the completion of $\mathbb{A}^{T}$ at $x$.
We note that this completion is exactly the same as the local morphism described in  (\ref{eq:localhitchbass}).

Secondly,  by Remark \ref{rmk:ademf},
the cokernel of  $g_{x}^{-1}\theta g_{x} \left(\in \pi_{x}^{-d_{x}}\mathfrak{gl}_n(\mathcal{O}_{x}\otimes_k R)\right)$ for $x\in \mathcal{S}\setminus \mathcal{S}^{\mathcal{F}}_1$ is isomorphic to $R[X]/(X^{n\cdot d_{x}+\ord_{x}(a_n)})$.
Thus the coefficients of $\chi_\theta$ are contained in the completion of $\mathbb{A}^{T}$ at $x$.
\end{proof}

We denote the restricted Hitchin fibration by $\Phi^T:\mathcal{M}^{T}\rightarrow \mathcal{A}^{T}$.
Abusing the notation, we also denote the restricted Hitchin fibration on $\mathbb{M}^{T}$ to $\mathbb{A}^{T}$ by $\Phi^T$.
This setting is visualized as follows:
\begin{equation}\label{diag:f'}
           \begin{tikzcd}[row sep=2.1em, column sep=5.5em]
        & \ar[ddl,  phantom, "\square" swap] \arrow[ld, "smooth" swap]\mathcal{M}^{T} \arrow[r, "\Phi^T"]    \arrow[d,hook,"\textit{locally closed}"]             & \mathcal{A}^{T} \arrow[d,hook,"closed"] \arrow[rd, "\textit{\'etale}"] \ar[ddr, phantom, "\square"] &\\
 \mathbb{M}^{T} \arrow[d,hook, "\textit{locally closed}"] &  \arrow[ld, "\textit{smooth}"]  \mathcal{M}^{ell} \arrow[r, "\Phi"]                &  \mathcal{A}^{ell} \arrow[rd, "\textit{\'etale}" swap] & \mathbb{A}^{T} \arrow[d,hook, "closed" ]\\ 
    \mathbb{M}^{ell}   & & &   \mathbb{A}^{\infty}  \arrow[r,   "open"] &\mathbb{A}.
        \end{tikzcd}
    \end{equation}
Here the left and the right square diagrams are Cartesian, whereas the middle is not.

\begin{remark}\label{rmk:stra}
Suppose that $n\cdot d_x+\ord_x(a_n)\leq 2$ for all $x\in \mathcal{S}$. 
Let $\mathcal{I}$ be the set of types $T$ of $\chi$ in the sense of Definition \ref{def:typeofchia}. 
Note that the set $\mathcal{I}$ is finite.   Then
an adelic description of an element in $\Phi^{-1}((\chi, \tau))(\chi)$ directly yields that  
\[
\Phi^{-1}((\chi, \tau))(k)=\bigsqcup_{T\in \mathcal{I}}(\Phi^T)^{-1}((\chi, \tau))(k) ~~~~  \textit{ and } ~~~~~ \Phi^{-1}(\chi)(k)=\bigsqcup_{T\in \mathcal{I}}(\Phi^T)^{-1}(\chi)(k).
\]
We call it \textbf{algebraic stratification}. 
In Section \ref{subsec:globalsmoo}, we will prove that the scheme $(\Phi^T)^{-1}((\chi, \tau))$ is smooth over $k$ for any $T\in \mathcal{I}$ (cf. Corollary \ref{cor:smoothstr}).
\end{remark}

\subsubsection{Algebraic stratification  with $n=2$}\label{subsecn=2}
We suppose that $n=2$ in this subsection.
For $\mathcal{F}$ as  in Definition \ref{def:coherentsheafF},  $\mathcal{F}_x\cong \Mfo_x/(\pi_x^{a_x})\oplus \Mfo_x/(\pi_x^{b_x})$ for $x\in \mathcal{S}$ such that $0\leq a_x\leq b_x$ and $a_x+b_x=2  d_x+\ord_x(a_2)$.  
We define a closed subscheme $\mathbb{A}^{\mathcal{F}}$ of $\mathbb{A}^\infty$ such that
$$\mathbb{A}^{\mathcal{F}}:=
   \left(H^0\Big(C,\mathcal{O}_C\big(D-\sum\limits_{x\in\mathcal{S}} a_{x}[x]\big)\Big)\bigoplus
   H^0\Big(C,\mathcal{O}_C\big(2D-\sum\limits_{x\in\mathcal{S}} (a_x+b_{x})[x]\big)\Big)
   \Big)\right)\cap \mathbb{A}^\infty.
   $$
Let $\mathcal{A}^{\mathcal{F}}:=\mathcal{A}^{ell}\times_{\mathbb{A}^\infty} \mathbb{A}^{\mathcal{F}}$.
Then we have the following result about the restricted Hitchin fibration.
\begin{lemma} 
The image of  $\Phi$ restricted to $\mathcal{M}^{\mathcal{F}}$ lies in 
$ \mathcal{A}^{\mathcal{F}}$.
\end{lemma}
\begin{proof}
The proof is parallel to the case with $x\in \mathcal{S}\backslash
 \mathcal{S}^{\mathcal{F}}_1$ in Lemma \ref{lem:rest}, and thus we skip it. 
 \end{proof}

\begin{remark}\label{rmk:typeforn=2}
    We denote the restricted Hitchin fibration by $\Phi^\mathcal{F}:\mathcal{M}^{\mathcal{F}}\rightarrow \mathcal{A}^{\mathcal{F}}$.
Abusing the notation, we also denote the restricted Hitchin fibration on $\mathbb{M}^{\mathcal{F}}$ to $\mathbb{A}^{\mathcal{F}}$ by $\Phi^\mathcal{F}$.
For synchronization with the case that $n\geq 3$, we say that $\mathcal{F}$ is a type $T$ of $\chi$ as in Definition \ref{def:typeofchia}, and thus we use  $\Phi^T:\mathcal{M}^{T}\rightarrow \mathcal{A}^{T}$ in place of $\Phi^\mathcal{F}:\mathcal{M}^{\mathcal{F}}\rightarrow \mathcal{A}^{\mathcal{F}}$.
Then Diagram \eqref{diag:f'} holds when $n=2$.
\end{remark}

\begin{remark}\label{rmk:stran=2}
Let $\mathcal{I}$ be the finite set of coherent sheaves  $\mathcal{F} (=\textit{type } T)$ associated to $\chi$  in Definition \ref{def:coherentsheafF}.
Then an adelic description of an element in $\Phi^{-1}((\chi, \tau))(\chi)$ directly yields that  
\[
\Phi^{-1}((\chi, \tau))(k)=\bigsqcup_{T\in \mathcal{I}}(\Phi^T)^{-1}((\chi, \tau))(k) ~~~~  \textit{ and } ~~~~~ \Phi^{-1}(\chi)(k)=\bigsqcup_{T\in \mathcal{I}}(\Phi^T)^{-1}(\chi)(k).
\]
We call it \textbf{algebraic stratification} when $n=2$. 
\end{remark}

\subsection{Global smoothness}\label{subsec:globalsmoo}
A goal of this subsection is to prove that the morphism $\Phi^T:\mathcal{M}^T\rightarrow \mathcal{A}^T$ (cf. Diagram (\ref{diag:f'}))  is smooth at any point in $(\Phi^T)^{-1}((\chi, \tau))(k)$ for a type $T$ (cf. Definition \ref{def:typeofchia} when $n\geq 3$ and Remark \ref{rmk:typeforn=2} when $n=2$).
Since it highly depends on the smoothness of local smoothness given in Section \ref{sec:localsmoothening}, we will start by introducing local objects associated to $f$ and the restricted Hitchin fibration $\Phi^T$.

\begin{itemize}
    \item For $x\in C(k)$, let $\mathbb{A}_x:=\bigoplus\limits_{i=1}^n \co_C(iD)_x$, where
    $\co_C(iD)_x$ is the completion of the stalk of the coherent sheaf $\co_C(iD)$ at $x$ so that  $\co_C(iD)_x=\pi_x^{-i\cdot d_x}\co_x$.

        \item Let $\mathbb{M}_x:=\pi_x^{-d_x} \gl_{n, \co_x}$, and let 
 $\Phi_x:\mathbb{M}_x\rightarrow \mathbb{A}_x$ be the local Hitchin fibration.

\item Let $\mathbb{A}_x^{T}$ be the `completion' of $\mathbb{A}^{T}$ at $x\in C$ in the above sense so that 
\[
\mathbb{A}_x^{T}=\left\{ \begin{array}{l l}
\bigoplus\limits_{i=1}^n\pi_x^{-i\cdot d_x}\co_x  
& \textit{if $x\notin \mathcal{S}$};\\
\bigoplus\limits_{i=1}^n\pi_x^{-i\cdot d_x+a_{i,x}+b_{i,x}}\co_x  
 & \textit{if $x\in \mathcal{S}$ and $n\geq 3$};\\
\pi_x^{-d_x+a_x}\co_x \bigoplus \pi_x^{-2d_x+a_x+b_x}\co_x 
 & \textit{if $x\in \mathcal{S}$ and $n=2$}.
    \end{array} \right.
\]
Here we refer to Section \ref{subsubsecn>3} (resp. Section \ref{subsecn=2}) for $a_{i,x}$ and $b_{i,x}$ (resp. $a_x$ and $b_x$).

\item For a variety $X$ over $k$, we say that $x_{m+1}\in X(k[\epsilon]/(\epsilon^{m+1}))$ is a lift of $x_{m}\in X(k[\epsilon]/(\epsilon^{m}))$ if the reduction of $x_{m+1}$ modulo $\epsilon^m$ is $x_m$.
\end{itemize}

From now on until the end of this subsection except for Theorem \ref{thm:smoon=2}, we will work with the case $n\geq 3$ for simplicity of the presentation of the paper.
Arguments for the case with $n=2$ are parallel to those for the case that $\mathcal{S}^{\mathcal{F}}_1=\emptyset$ with $n\geq 3$, and thus we skip details but state the main result in Theorem \ref{thm:smoon=2}.
Let $R=k[\epsilon]/(\epsilon^m)$. 
To emphasize the role of $m$, we denote 
\[
\left\{ \begin{array}{l}
\textit{$(g_m,\theta_m)\left(=\left((g_{x,m})_{x\in C}, \theta_m\right)\right)$: an element of  $\mathbb{M}^{T}(k[\epsilon]/(\epsilon^m))$ with $m\geq 1$};\\
\textit{$\chi_m\in \mathbb{A}^{T}(k[\epsilon]/(\epsilon^m))$: the characteristic polynomial of $\theta_m$ such that $\chi_m\equiv \chi$ modulo $\epsilon$}
    \end{array} \right.
\]

Here $\chi_m$ is also an element of $\mathbb{A}_x^{T}(k[\epsilon]/(\epsilon^m))$ for $x\in C$.
Recall that $g_{x,m}^{-1}\theta_m g_{x,m}$ is contained in $\pi_x^{-d_x}\gl_n(\mathcal{O}_x[\epsilon]/(\epsilon^{m}))$ by Remark \ref{rmk:adelic}, and satisfies the conditions stated in Remarks \ref{rmk:ademf} and \ref{rmk:ademfly}.

\begin{lemma}[Existence of a local lift]\label{lem:localsmoothness}
    For a lift $\chi_{m+1}$ in $\mathbb{A}^{T}(k[\epsilon]/(\epsilon^{m+1}))$ of $\chi_m$,
there exists an element $X_{x,m+1}$ in $\pi_x^{-d_x}\gl_n(\mathcal{O}_x[\epsilon]/(\epsilon^{m+1}))$ satisfying the following conditions:
    \begin{enumerate}[(a)]
        \item The reduction of $X_{x,m+1}$ modulo $\epsilon^m$ is  $g_{x,m}^{-1}\theta_m g_{x,m}$;
        \item The characteristic polynomial of $X_{x,m+1}$ is $\chi_{m+1}$;
        \item The cokernel of $X_{x,m+1}$ in $\pi_x^{-d_x}\gl_n(\Mfo_x[\epsilon]/(\epsilon^{m+1}))$ is isomorphic to $\mathcal{F}_x\otimes_k k[\epsilon]/(\epsilon^{m+1})$;
        \item If $x\in \mathcal{S}^{\mathcal{F}}_1$, then the reduction of $\pi_x^{d_x}X_{x,m+1}$ modulo $\pi_x$ is conjugate to $J_{l_x+1,n-l_x-1}$ by an element of $\mathrm{GL}_n(k[\epsilon]/(\epsilon^{m+1}))$.
    \end{enumerate}
\end{lemma}

\begin{remark}
Before providing the proof, let us explain what the above four conditions mean.
\begin{enumerate}
    \item The condition (c) comes from the right hand side of Remark \ref{rmk:ademf}, which explains an adelic description of when an element of $ \mathbb{M}(k[\epsilon]/(\epsilon^{m+1}))$ is contained in $ \mathbb{M}^{\mathcal{F}}(k[\epsilon]/(\epsilon^{m+1}))$.

    \item The condition (d) comes from the right hand side of Remark \ref{rmk:ademfly}, which explains an adelic description of when an element of $ \mathbb{M}^{\mathcal{F}}(k[\epsilon]/(\epsilon^{m+1}))$ is contained in $ \mathbb{M}^{\mathcal{F}, l_x}(k[\epsilon]/(\epsilon^{m+1}))$.

    \item The conditions (a) and (b) mean that $X_{x,m+1}$ is a lift of $g_{x,m}^{-1}\theta_m g_{x,m}$ according to the local Hitchin morphism stated in Corollaries \ref{cor:genlocalthm1} and \ref{cor:genlocalthm}, since $\Mfo_x[\epsilon]/(\epsilon^m)$ is a flat $\Mfo_x$-algebra so that the local Hitchin morphism is to assign the coefficients of the characteristic polynomial. 

\end{enumerate}



\end{remark}

\begin{proof}
 Suppose that $x \notin \mathcal{S}^{\mathcal{F}}_1$.
Then $g_{x,m}^{-1}\theta_m g_{x,m}$ and $\chi_m$ are $k[\epsilon]/(\epsilon^m)$-points of both sides of the morphism given in (\ref{eq:localhitch}) with $t=-d_x$. 
Note that  $M=L$ if $x\notin \mathcal{S}$. 
The existence of $X_{x, m+1}$ follows from the formal smoothness of the morphism by Corollary \ref{cor:genlocalthm1}.

Suppose that $x\in \mathcal{S}^{\mathcal{F}}_1$. 
Then, by Lemma \ref{lem:fatpoints}, $g_{x,m}^{-1}\theta_mg_{x,m}$ has a representative as an $\Mfo_x[\epsilon]/(\epsilon^m)$-point of $\mathrm{End}_{\Mfo_x}(L)_{M, l_x}$ given in (\ref{eq:localhitchinbasst}) with $t=-d_x$. 
The existence of $X_{x, m+1}$ follows from the formal smoothness of $\varphi_{l_x}^{-d_x}$ in Corollary \ref{cor:genlocalthm}.
Here, the verification of the condition (d) for $X_{x,m+1}$ is explained in the 'if' part of the proof of Lemma \ref{lem:fatpoints}.
\end{proof}



\begin{proposition}[Existence of a global lift]\label{prop:kmpointslift}
For a lift $\chi_{m+1}\in \mathbb{A}^{T}(k[\epsilon]/(\epsilon^{m+1}))$ of $\chi_m$, any point in  $(\Phi^T)^{-1}(\chi_m)$, which is contained in $\mathbb{M}^{T}(k[\epsilon]/(\epsilon^m))$, lifts to a point in $\mathbb{M}^{T}(k[\epsilon]/(\epsilon^{m+1}))$ whose characteristic polynomial is $\chi_{m+1}$.




\end{proposition}
\begin{proof}


Since $\chi_m$ is the characteristic polynomial of $\theta_m$, 
\cite[Theorem III.2]{Mc78} yields that $\theta_m$ is conjugate to a companion matrix of the polynomial $\chi_m$ by an element of $\mathrm{GL}_n(F[\epsilon]/(\epsilon^m))$.
Thus, by applying the left action of $\mathrm{GL}_n(F[\epsilon]/(\epsilon^m))$ on $(g_m,\theta_m)$ (cf. Remark \ref{rmk:adelic}), we may assume that $\theta_m$ is the companion matrix of $\chi_m$.
Let $\theta_{m+1}\in \mathfrak{gl}_n(F[\epsilon]/(\epsilon^{m+1}))$ be the companion matrix of a polynomial $\chi_{m+1}$.
We then aim to find a lift $g_{m+1}$ of $g_m$ in $\mathrm{GL}_n(\mathbb{A}_C[\epsilon]/(\epsilon^{m+1}))$ such that $(g_{m+1},\theta_{m+1})$ is a desired lift in $\mathbb{M}^{T}(k[\epsilon]/(\epsilon^{m+1}))$.

Lemma \ref{lem:localsmoothness} guarantees the existence of  a (local) lift $X_{x,m+1}$ of $g_{x,m}^{-1}\theta_m g_{x,m}$ in $\pi_x^{-d_x}\mathfrak{gl}_n(\mathcal{O}_x[\epsilon]/(\epsilon^{m+1}))$ with characteristic polynomial $\chi_{m+1}$ satisfying the cokernel condition (cf. Lemma \ref{lem:localsmoothness}.(c)) and the Jordan block condition (cf. Lemma \ref{lem:localsmoothness}.(d)).
Since $\theta_{m+1}$ and $X_{x,m+1}$ have the same characteristic polynomial $\chi_{m+1}$, \cite[Theorem III.2]{Mc78} yields the existence of an element $h_{x,m+1}\in \mathrm{GL}_n(F_x[\epsilon]/(\epsilon^{m+1}))$ such that $X_{x,m+1}=h_{x,m+1}^{-1}\theta_{m+1}h_{x,m+1}$.

Let $h_{x,m}$ be the reduction of $h_{x,m+1}$ modulo $\epsilon^m$. 
Then  $g_{x,m}^{-1}\theta_mg_{x,m}=h_{x,m}^{-1}\theta_m h_{x,m}$ so that the element $g_{x,m}h_{x,m}^{-1}$ in $\mathrm{GL}_n(F_x[\epsilon]/(\epsilon^m))$ centralizes $\theta_m$.
Denote $c_{x,m}:=g_{x,m}h_{x,m}^{-1}$.
We then use Lemma \ref{lemma:centralizer}, which will be stated and proved below, to guarantee the existence of a lift $c_{x,m+1}\in \mathrm{GL}_n(F_x[\epsilon]/(\epsilon^{m+1}))$ of $c_{x, m}$ such that $c_{x,m+1}$ stabilizes $\theta_{m+1}$ in $\mathrm{GL}_n(F_x[\epsilon]/(\epsilon^{m+1}))$.
Define $g_{x,m+1}:=c_{x,m+1}h_{x,m+1}$ in $\mathrm{GL}_n(F_x[\epsilon]/(\epsilon^{m+1}))$.
We finally claim that $(g_{m+1}, \theta_{m+1})\left(=\left(g_{x,m+1})_{x\in C}, \theta_{m+1}\right)\right)$ is a desired lift of $(g_{m}, \theta_{m})$.

Firstly, the reduction of $g_{x,m+1}$ modulo $\epsilon^m$ is equal to $c_{x,m}h_{x,m}=g_{x,m}$, and we have $$g_{x,m+1}^{-1}\theta_{m+1}g_{x,m+1}
=h_{x,m+1}^{-1}\theta_{m+1}h_{x,m+1}=X_{x,m+1}  \in \pi_x^{-d_x}\mathfrak{gl}_n(\mathcal{O}_x[\epsilon]/(\epsilon^{m+1}))$$
so that $g_{x,m+1}^{-1}\theta_{m+1}g_{x,m+1}$ satisfies all conditions (a)-(d) of Lemma \ref{lem:localsmoothness}.
Therefore, it remains to show that $g_{m+1}=(g_{x,m+1})_{x\in C}$ is a well-defined element of $\mathrm{GL}_n(\mathbb{A}_C[\epsilon]/(\epsilon^{m+1}))$. 

We choose the set of $x$'s such that  $x\notin \mathcal{S}\cup \mathrm{Supp}(D)$, $g_{x,m}=1$, and $\theta_{m+1}\in \gl_n(\mathcal{O}_x[\epsilon]/(\epsilon^{m+1}))$. 
At such point $x$, the companion matrix $\theta_{m+1}$ satisfies conditions (a)-(c) in Lemma \ref{lem:localsmoothness}. Here the condition (d) is unnecessary. 
Thus we  put $X_{x,m+1}=\theta_{m+1}$ at such point $x$ so that $h_{x,m+1}=1$ and  $c_{x,m}=g_{x,m}=1$.
We now put $c_{x,m+1}=1$, which satisfies the condition of Lemma \ref{lemma:centralizer}. 
Then $g_{x,m+1}=1$, which completes the proof. 
\end{proof}

\begin{lemma}\label{lemma:centralizer}
    Choose a lift $\theta_{m+1}$  in $\mathfrak{gl}_n(F[\epsilon]/(\epsilon^{m+1}))$ of $\theta_m$.
    If  $c_{x,m}$ in $\mathrm{GL}_n(F_{x}\otimes_k k[\epsilon]/(\epsilon^{m}))$ centralizes $\theta_m$, then there exists a lift $c_{x,m+1}$ in $\mathrm{GL}_n(F_{x}[\epsilon]/(\epsilon^{m}))$ of $c_{x,m}$ that centralizes $\theta_{m+1}$.
\end{lemma}
\begin{proof}
We write  $\theta_m=\theta_m+0\cdot \epsilon^m$ and $c_{x,m}=c_{x,m}+0\cdot \epsilon^m$ so that they are viewed as elements of $\mathfrak{gl}_n(F[\epsilon]/(\epsilon^{m+1}))$ and $\mathrm{GL}_n(F_{x}[\epsilon]/(\epsilon^{m+1}))$, respectively. 
Then  $\theta_{m+1}$  is of the form $\theta_m+\epsilon^m\alpha$ with $\alpha\in \mathfrak{gl}_n(F)$
and $c_{x,m+1}$  is of the form $c_{x,m}+\epsilon^m\beta$ with $\beta\in \gl_n(F_x)$.
Thus it suffices to show the existence of a solution  $\beta\in \mathfrak{gl}_n(F_x)$ satisfying the following matrix equation:
    $$(\theta_m+\epsilon^m\alpha)(c_{x,m}+\epsilon^m\beta)=(c_{x,m}+\epsilon^m\beta)(\theta_m+\epsilon^m\alpha) \textit{ in } \mathfrak{gl}_n(F_x[\epsilon]/(\epsilon^{m+1})).$$
    Let $\theta_0\in \mathfrak{gl}_n(F)$ and $c_{x,0}\in \mathrm{GL}_n(F_x)$ be the reduction of $\theta_m$ and $c_{x,m}$ modulo $\epsilon$, respectively.
    Since $\epsilon^m\theta_m=\epsilon^m\theta_0$ and $\epsilon^m c_{x,m}=\epsilon^m c_{x,0}$, the above matrix equation is equivalent to the following:
\begin{equation}\label{eq:lineareq}
    \epsilon^m(\beta\theta_0-\theta_0\beta)=\theta_mc_{x,m}-c_{x,m}\theta_m+\epsilon^m(\alpha c_{x,0}-c_{x,0}\alpha) \textit{ in } \mathfrak{gl}_n(F_x[\epsilon]/(\epsilon^{m+1})).
    \end{equation}
    Since $c_{x,m}$ centralizes $\theta_m$ in $\mathfrak{gl}_n(F_x[\epsilon]/(\epsilon^m))$, the element  $\theta_mc_{x,m}-c_{x,m}\theta_m$ is contained in $\epsilon^m\mathfrak{gl}_n(F_x[\epsilon]/(\epsilon^{m+1}))$.
    Thus, Equation (\ref{eq:lineareq}) divided by $\epsilon^m$ is a linear matrix equation with coefficients in $F_x$.
    Now, it suffices to prove that Equation (\ref{eq:lineareq}) divided by $\epsilon^m$ has a solution $\beta$ in $\mathfrak{gl}_n(F_x)$. 
    This is equivalent to showing the existence of a solution $\beta$ in $\mathfrak{gl}_n(\overline{F_x})$, where $\overline{F_x}$ is the algebraic closure of $F_x$.

Let $\chi_{m+1}\in F_x[\epsilon]/(\epsilon^{m+1})[X]$ be the characteristic polynomial of  $\theta_{m+1}$.
We view $\chi_{m+1}$ as an element of $\overline{F_x}[[\epsilon]][X]$, by assigning $0$ to the coefficients of $\epsilon^i$ with $i> m$.
Since $\chi_{m+1}$ modulo $\epsilon$ is the characteristic polynomial  of $\theta_0$, it is separable in $\overline{F_x}[X]$.
Using Hensel's Lemma, $\chi_{m+1}$  splits into distinct linear factors in $\overline{F_x}[[\epsilon]][X]$, and thus it splits into distinct linear factors in $\left(\overline{F_x}[\epsilon]/(\epsilon^{m+1})\right)[X]$.
    Therefore, \cite[Theorem III.2]{Mc78} is applicable which guarantees the existence of a matrix $P_x\in \mathrm{GL}_n(\overline{F_x}[\epsilon]/(\epsilon^{m+1}))$ such that $P_x\theta_{m+1}P_x^{-1}$ is a diagonal matrix in $\gl_n(\overline{F_x}[\epsilon]/(\epsilon^{m+1}))$. 
    Thus the reduction of $P_x\theta_{m+1}P_x^{-1}$ modulo $\epsilon^m$ is also a diagonal matrix in $\gl_n(\overline{F_x}[\epsilon]/(\epsilon^{m}))$.

Now, since the reduction of $P_x(c_{x,m}+0\cdot \epsilon^m)P_x^{-1}$ modulo $\epsilon^m$ centralizes the reduction of $P_x\theta_{m+1}P_x^{-1}$ modulo $\epsilon^m$ and the reduction of $P_x\theta_{m+1}P_x^{-1}$ modulo $\epsilon$ has distinct diagonal entries in $\overline{F_x}$, we claim that the reduction of $P_x(c_{x,m}+0\cdot \epsilon^m)P_x^{-1}$ modulo $\epsilon^m$  is also a diagonal matrix in $\gl_n(\overline{F_x}[\epsilon]/(\epsilon^{m}))$.

The above claim is justified as follows. 
Consider a diagonal matrix $A\in \gl_n(\overline{F_x}[\epsilon]/(\epsilon^{m}))$ which has distinct diagonal entries modulo $\epsilon$ in $\overline{F}_x$. Choose $B\in \gl_n(\overline{F_x}[\epsilon]/(\epsilon^{m}))$ such that $AB=BA$.
Let $b_{i,j}\in \overline{F_x}[\epsilon]/(\epsilon^{m})$ be the $(i,j)$-entry of $B$ and let $\lambda_i\in \overline{F_x}[\epsilon]/(\epsilon^{m})$ be the $i$-th diagonal entry of $A$.
The equation $AB=BA$ yields that $(\lambda_i-\lambda_j)b_{i,j}=0$. 
Here  $\lambda_i-\lambda_j$  is a unit in $\overline{F_x}[\epsilon]/(\epsilon^{m})$ if $i\neq j$ since $\lambda_i \not\equiv \lambda_j$ modulo $\epsilon$. This directly yields that $b_{i,j}=0$ if $i\neq j$ so that $B$ is diagonal.

    Therefore, taking the conjugation by $P_x$ to Equation (\ref{eq:lineareq}), 
    we may assume $\theta_m, \theta_0,\alpha$, and $c_{x,0}$ to be diagonal and may write $c_{x,m}=D_{x,m}+\epsilon^m\beta'$, where $D_{x,m}$ is a diagonal matrix in $\mathfrak{gl}_n(\overline{F_x}[\epsilon]/(\epsilon^{m+1}))$ with no $\epsilon^m$-term.
    Here $\beta'\in \gl_n(\overline{F_x})$ does not need to be diagonal. 
    Then Equation (\ref{eq:lineareq}) is equivalent to the following matrix equation:
    $$\beta\theta_0-\theta_0\beta=\theta_0\beta'-\beta'\theta_0 \textit{ in } \mathfrak{gl}_n(\overline{F_x}).$$
    This matrix equation has a solution $\beta=-\beta'$.
\end{proof}

We now prove a global geometrization of Theorems \ref{thm:localtype1line} and \ref{thm:localBassmain} when $n\geq 3$, mainly based on Proposition \ref{prop:kmpointslift}.
    Here, we recall that $a=(\chi,\tau)$ is a point of $\mathcal{A}(k)$ which lies over $\chi$.

\begin{theorem}\label{thm:smoothness}
For a type $T=(\mathcal{F}, (l_x)_{x\in \mathcal{S}^{\mathcal{F}}_1})$ of $\chi$,        $\mathcal{M}^{T}$ is nonsingular at every point of $(\Phi^T)^{-1}((\chi, \tau))(\bar{k})$ and the restricted Hitchin fibration $\Phi^T:\mathcal{M}^{T}\rightarrow \mathcal{A}^{T}$ is smooth at every point in $(\Phi^T)^{-1}((\chi, \tau))(\bar{k})$.
\end{theorem}
\begin{proof}
   We first prove nonsingularity of $\mathcal{M}^{T}$ at a point in $(\Phi^T)^{-1}((\chi, \tau))(\bar{k})$.
Here $\mathcal{M}^{T}$ is a scheme over ${k}$ of finite type so as to be a flat ${k}$-scheme (cf. Definition \ref{def:constmfl}). Then \cite[Corollary 1.10]{Win77} is applicable so that it suffices to show 
formal smoothness with respect to small extensions of the form $\mathrm{Spec}~ \bar{k}[\epsilon]/(\epsilon^m)\rightarrow \mathrm{Spec}~\bar{k}[\epsilon]/(\epsilon^{m+1})$.
Since the morphism $\mathcal{M}^{ell}\rightarrow \mathbb{M}^{ell}$ is smooth (cf. Diagram \eqref{diag:1st}) so that the morphism $\mathcal{M}^{T} \rightarrow  \mathbb{M}^{T}$ is smooth as well (cf. Diagram \eqref{diag:f'}),  it suffices to show formal smoothness of $\mathbb{M}^{T}$ at points in $(\Phi^T)^{-1}(\chi)$.
    This is a direct consequence of Proposition \ref{prop:kmpointslift}.
    
For the smoothness of $\Phi^T$ at a point in  $(\Phi^T)^{-1}((\chi, \tau))(\bar{k})$, it suffices to show the surjectivity on the Zariski tangent space by \cite[Proposition 10.4]{Har77} since $\mathcal{M}^{T}$ is nonsingular at every point of $(\Phi^T)^{-1}((\chi, \tau))(\bar{k})$.
This is a special case of Proposition \ref{prop:kmpointslift} with $m=1$.
\end{proof}

\begin{corollary}[\textbf{Algebraic stratification of $\Phi^{-1}((\chi, \tau))$}]\label{cor:smoothstr}
Suppose that $\ord_x(a_n)+n\cdot d_x\leq 2$ for all $x\in \mathcal{S}$.
Then $(\Phi^T)^{-1}((\chi, \tau))$ is a locally closed and smooth subvariety of $\Phi^{-1}((\chi, \tau))$ for a type $T$ of $\chi$ (cf. Definition \ref{def:typeofchia}). 
In addition, $\Phi^{-1}((\chi, \tau))(k')$ admits the following stratification:
\[
\Phi^{-1}((\chi, \tau))(k')=\bigsqcup_{T\in \mathcal{I}}(\Phi^T)^{-1}((\chi, \tau))(k') ~~~~  \textit{  for any extension $k'/k$ in $\bar{k}$},
\]
where the index set $\mathcal{I}$ is the set of types $T$ of $\chi$.

\end{corollary}

\begin{proof}
    By Diagram (\ref{diag:f'}), $(\Phi^T)^{-1}((\chi, \tau))$ is a locally closed subvariety of $\Phi^{-1}((\chi, \tau))$, and is smooth by Theorem \ref{thm:smoothness}. 
   The stratification follows from Remark \ref{rmk:stra}.
\end{proof}

\begin{remark}
The Hitchin fibration
$\Phi: \mathcal{M}^{ell}\longrightarrow  \mathcal{A}^{ell}$ is proper by Faltings (cf. \cite[Theorem 6.1.1]{Ch14}).
Thus $\Phi^{-1}((\chi, \tau))$ is proper over $k$, whereas $(\Phi^T)^{-1}((\chi, \tau))$ is not necessarily proper.
\end{remark}

All arguments used in Theorem \ref{thm:smoothness} when $n\geq 3$ with the restriction that $\mathcal{S}^{\mathcal{F}}_1=\emptyset$ hold when $n=2$. Thus we state a global geometrization with $n=2$ below, without proof.

\begin{theorem}\label{thm:smoon=2}
Suppose that $n=2$. 
Let $T$ be a type of $\chi$ in Remark \ref{rmk:stran=2} and $\mathcal{I}$ be the set of $T$'s.
\begin{enumerate}
    \item 
 $\mathcal{M}^{T}$ is nonsingular at every point of $(\Phi^T)^{-1}((\chi, \tau))(\bar{k})$ and the restricted Hitchin fibration $\Phi^T:\mathcal{M}^{T}\rightarrow \mathcal{A}^{T}$ is smooth at every point in $(\Phi^T)^{-1}((\chi, \tau))(\bar{k})$.

\item{\textbf{Algebraic stratification of $\Phi^{-1}((\chi, \tau))$}} 
$(\Phi^T)^{-1}((\chi, \tau))$ is a locally closed and smooth subvariety of $\Phi^{-1}((\chi, \tau))$ and $\Phi^{-1}((\chi, \tau))(k')$ admits the following stratification:
\[
\Phi^{-1}((\chi, \tau))(k')=\bigsqcup_{T\in \mathcal{I}}(\Phi^T)^{-1}((\chi, \tau))(k') ~~~~  \textit{  for any extension $k'/k$ in $\bar{k}$}.
\]
\end{enumerate}
\end{theorem}

\section{Comparison between two stratifications: Bass case}\label{Section:geomstr}
In this section, we introduce another description of the Hitchin fiber $\Phi^{-1}(\chi)$ and $\Phi^{-1}((\chi, \tau))$, using the \textit{spectral curve} $Y_\chi$, following \cite{Ch14}.
In the geometric sense, each point in $\Phi^{-1}(\chi)$ corresponds to a rank-1 torsion-free sheaf over $Y_\chi$.
The moduli space of these rank-1 torsion-free sheaves, denoted by $\overline{\mathrm{Pic}^0}(Y_\chi)$, allows the geometric stratification according to \textit{partial normalizations} of $Y_\chi$.

The main goal is Theorem \ref{thm:eqoftwostra}, which shows that the geometric stratification agrees with the algebraic stratification of $\Phi^{-1}(\chi)$ as schemes. 
Throughout this section, we assume that $k=\bar{k}$.


\subsection{Spectral curve and the Hitchin fiber}
We continue using the notation used in Section \ref{section3}.
We first review the definition of a spectral curve, which is taken from \cite[Sections 4.2]{Ch14}. 
Let 
$$\pi_\Sigma: \Sigma_D=\underline{\mathrm{Spec}}(\bigoplus\limits_{i=0}^\infty \Mfo_C(-iD)X^i)\rightarrow C$$
be the total space of the line bundle $\Mfo_C(D)$, which is the relative spectrum of the symmetric algebra of its dual. 
The \textit{spectral curve} $Y_\chi$ associated 
to $\chi\in \mathbb{A}^\infty(k)$ is the closed curve in $\Sigma_D$ defined by 
$$Y_\chi=\underline{\mathrm{Spec}}((\bigoplus\limits_{i=0}^\infty \Mfo_C(-iD)X^i)/\mathcal{I}_\chi),$$
where the ideal sheaf $\mathcal{I}_\chi$ is generated by $\Mfo_C(-nD)\chi$.
By the construction, there is a canonical projection $\pi_\chi: Y_\chi\rightarrow C$, which is a finite cover of degree $n$.
Since $\chi\in \mathbb{A}^\infty(k)$, $\pi_\chi$ is \'etale over $\infty$.
If we consider $(\chi,\tau)\in \mathcal{A}^{ell}(k)$, 
then the spectral curve $Y_{(\chi,\tau)}$ is defined in the same way as $Y_\chi$, but it comes with an additional datum of an ordering on the points of $\pi_\chi^{-1}(\infty)$: $\pi_\chi^{-1}(\infty)=\{\infty_1, \dots, \infty_n\}$.
\begin{remark}\label{remark41}
    By the definition of the relative spectrum, if $U$ is an affine open subset of $C$, then there exists an isomorphism  which is compatible with a restriction described below:  
    $$\pi_\chi^{-1}(U)\cong \mathrm{Spec}(H^0(U,\bigoplus\limits_{i=0}^\infty \Mfo_C(-iD)X^i/\mathcal{I}_\chi)).$$ 
\end{remark}
\begin{proposition}\label{prop:spectralcurvegeometry}
    \begin{enumerate}
        \item $Y_\chi$ is an integral projective curve.
        \item For any $x\in C$, the completed local ring of $Y_\chi$ at a point in $\pi_\chi^{-1}(x)$ is isomorphic to the localization of $\Mfo_x[X]/(\chi_x(X))$ at a maximal ideal.
Here $\chi_x(X)$ is given in Setting \ref{settings}. 
        \item A singularity of $Y_\chi$ lies over a point contained in $\mathcal{S}$. 
    \end{enumerate}
\end{proposition}
\begin{proof}
    By \cite[Proposition 4.2.1]{Ch14}, $Y_\chi$ is irreducible and reduced, thus integral.
    Projectivity of $Y_\chi$ is immediate from the fact that $\pi_\chi$ is a finite morphism.
    For (2), Remark \ref{remark41} yields that $\Mfo_x[X]/(\chi_x(X))$ is isomorphic to the stalk completion of $\bigoplus\limits_{i=0}^\infty \Mfo_C(-iD)X^i/\mathcal{I}_\chi$ at $x$.

For (3), suppose that $x\notin \mathcal{S}$.
Then $\overline{\chi_x(X)}$ is separable over $k$.
    Since $k=\bar{k}$,  Hensel's lemma implies that $\chi_x(X)$ splits into the product of $n$ coprime monic polynomials of degree $1$. 
    By Chinese Remainder Theorem, $\Mfo_x[X]/(\chi_x(X))\cong \Mfo_x^n$.
    Thus any point of $Y_\chi$ lying over $x$ is smooth.   
\end{proof}
Now we state the theorem which relates the spectral curve and the Hitchin fiber.
\begin{theorem}\label{thm:spectralcorr}
    \begin{enumerate}
        \item \cite[Section 4.4.1]{Ngo10}
        The Hitchin fiber $\Phi^{-1}(\chi)$ in $\mathbb{M}^\infty$ is isomorphic to the stack $\overline{\mathrm{Pic}^0}(Y_\chi)$ of rank-1 torsion-free coherent $\Mfo_{Y_\chi}-$modules $\mathscr{F}$ of degree $0$, up to isomorphism.
        \item \cite[Theorem 4.3.1]{Ch14} 
        The Hitchin fiber $\Phi^{-1}((\chi, \tau))$ in $\mathcal{M}^{ell}$ is isomorphic to the stack of rank-1 torsion-free coherent $\Mfo_{Y_\chi}-$modules $\mathscr{F}$ of degree $0$ up to isomorphism, equipped with a trivialization of their stalk at $\infty_1$.
    \end{enumerate}
\end{theorem}
Here, the degree of $\mathscr{F}$ is defined by using the Euler characteristic function, denoted by $\mathrm{Ec}$:
    $$\deg\mathscr{F}:=\mathrm{Ec}(Y_\chi,\mathscr{F})-\mathrm{Ec}(Y_\chi,\Mfo_{Y_\chi}), \textit{  as in the case of line bundles.}$$ 
The stack $\overline{\mathrm{Pic^0}}(Y_\chi)$ is called the \textit{compactified Jacobian} of $Y_\chi$.
In our case, it is a scheme.
\begin{corollary}\label{cor:repreofminfty}
    The Hitchin fiber $\Phi^{-1}(\chi)\left(\cong \overline{\mathrm{Pic}^0}(Y_\chi)\right)$ in $\mathbb{M}^\infty$ is representable by a scheme.
\end{corollary}
\begin{proof}
    Since $Y_\chi$ is an integral projective curve by Proposition \ref{prop:spectralcurvegeometry}, \cite[Theorem 8.1]{AK80} yields the representability of the moduli space of rank-1 torsion-free sheaves on $Y_\chi$ with a fixed Euler characteristic.
    Since Euler characteristic and degree of rank-1 torsion-free sheaves are in 1-1 correspondence, this concludes the proof.
\end{proof}

Then Corollary \ref{cor:smoothstr} and Theorem \ref{thm:smoon=2} are reformulated as follows:

\begin{corollary}[\textbf{Algebraic stratification of $\Phi^{-1}(\chi)$}]\label{cor:algstraoff-1a}
Suppose that $n=2$  or $\ord_x(a_n)+n\cdot d_x\leq 2$ for all $x\in \mathcal{S}$ when $n\geq 3$.
Then $(\Phi^T)^{-1}(\chi)$ is a locally closed and smooth subvariety of $\Phi^{-1}(\chi)$ for a type $T$ of $\chi$ (cf.  Definition \ref{def:typeofchia} when $n\geq 3$ and Remark \ref{rmk:typeforn=2} when $n=2$). 
In addition, $\Phi^{-1}(\chi)(k)$ admits the following stratification, where the index set $\mathcal{I}$ is the set of types $T$ of $\chi$:
\[
\Phi^{-1}(\chi)(k)=\bigsqcup_{T\in \mathcal{I}}(\Phi^T)^{-1}(\chi)(k).
\]
\end{corollary}


\subsection{Geometric stratification of $\overline{\mathrm{Pic^0}}(Y_\chi)$}\label{subsec:geomstr}
Theorem \ref{thm:spectralcorr} yields a tool to investigate the geometric structure of the Hitchin fiber via the compactified Jacobian.
For instance, if $Y_\chi$ is smooth, then $\overline{\mathrm{Pic^0}}(Y_\chi)$ is the Jacobian variety of $Y_\chi$, which is an abelian variety.
For a general $\chi$, the smooth locus of $\overline{\mathrm{Pic^0}}(Y_\chi)$ is well-known as explained below. 
\begin{theorem}\label{thm:smoothlocus}
\cite[Theorem 2.3]{MRV17}
The smooth locus of $\overline{\mathrm{Pic^0}}(Y_\chi)$ is equal to the generalized Jacobian $\mathrm{Pic}^0(Y_\chi)$, which is a commutative group variety parametrizing line bundles $\mathcal{L}$ of $Y_\chi$ of degree $0$. 
\end{theorem}
\begin{remark}\label{rmk:jac}
    \begin{enumerate}
    \item The generalized Jacobian $\mathrm{Pic}^0(Y_\chi)$ acts on $\overline{\mathrm{Pic}^0}(Y_\chi)$ via the tensor product:
    $$(\mathcal{L},\mathscr{F})\mapsto \mathcal{L}\otimes_{\Mfo_{Y_\chi}}\mathscr{F}.$$
    \item{\cite[Section 9.2, Corollary 11]{BRL}} 
   $\mathrm{Pic}^0(Y_\chi)$ satisfies the following short exact sequence:
    $$0\rightarrow A\rightarrow \mathrm{Pic^0}(Y_\chi) \rightarrow \mathrm{Pic}^0(\widetilde{Y}_\chi)\rightarrow 0,$$
    where $\widetilde{Y}_\chi$ is the normalization of $Y_\chi$ and $A$ is a smooth connected linear algebraic group. 
    
    \end{enumerate}
\end{remark}
From now on until the end of this section, we suppose that $n=2$ or $n\cdot d_{x}+\ord_{x}(a_n)\leq 2$ for any $x\in \mathcal{S}$ if $n\geq 3$. 
This assumption specifies the type of singularities of $Y_\chi$.
\begin{proposition}\label{prop:Bass}
The following statements are equivalent. 
\begin{enumerate}
    \item Either $n=2$ or $n\cdot d_x+\ord_x(a_n)\leq 2$ with any $x\in \mathcal{S}$.
    \item The completed local ring at any point of $Y_\chi$ is a \textit{Bass order} (cf. Definition \ref{def:bass}).
    \item The singularities of $Y_\chi$ are double points.
\end{enumerate}
\end{proposition}
\begin{proof}
    Firstly, we prove that (1) is equivalent to (2).
If $y\notin \pi_\chi^{-1}(\mathcal{S})\subset Y_\chi$, then the completed local ring at $y$ is $\Mfo_{\pi_\chi(y)}$ by Proposition \ref{prop:spectralcurvegeometry}.
    It is a DVR and thus a Bass order. 
    For $y\in \pi_\chi^{-1}(\mathcal{S})$, let $x=\pi_\chi(y)$.
    Then $\overline{\chi_x(X)}=X^n$ in $k[X]$ since $n\cdot d_x+\ord_x(a_n)>0$ and $\chi_x(X)$ is irreducible by the assumption. 
    If $n\cdot d_x+\ord_x(a_n)=1$, then $\chi_x(X)$ is an Eisenstein polynomial, and thus  $\Mfo_x[X]/(\chi_x(X))$ is a DVR and thus a Bass order. 
    If $n=2$ or $n\cdot d_x+\ord_x(a_n)>1$, the equivalence between (1) and (2) follows from \cite[Proposition 3.6]{CHL}.

    Now we prove the equivalence of (2) and (3).
     By condition 2 of \cite[Proposition 3.1]{Gau97}, (3) is equivalent to the condition that $\dim_k(I/\mathfrak{m}_y I) \leq 2$ for every point $y \in Y_\chi$ and for every ideal $I$ in the local ring $\mathcal{O}_{Y_\chi,y}$,  where $\mathfrak{m}_y$ denotes the maximal ideal of $\mathcal{O}_{Y_\chi,y}$.
    It suffices to verify this property on the completion $\widehat{\Mfo}_{Y_\chi,y}$ of $\mathcal{O}_{Y_\chi,y}$.
    If (2) is true, then $\widehat{\Mfo}_{x_\chi,y}$ is a Bass order, which implies that every ideal $I$ in the completion is generated by at most two elements.
    Therefore, the dimension of the quotient $I/\mathfrak{m}_xI$ is at most $2$, yielding (3).
    Conversely, if the dimension of $I/\mathfrak{m}_xI$ is at most $2$, then $I$ can be generated by two elements, by Nakayama's lemma.
\end{proof}

For a curve whose singularities are double points, the compactified Jacobian is well-studied in \cite{Gau97}.
To study the compactified Jacobian in our context of $Y_\chi$, we consider intermediate curves that resolve some, but not necessarily all, of the singularities. This is given in the following definition.

\begin{definition}\cite[page 22]{Gau97}
    A partial normalization of $Y_\chi$ is a finite surjective birational morphism $\pi':Y_\chi'\rightarrow Y_\chi$ from a curve $Y_\chi'$ over $k$.
    Two partial normalizations $\pi':Y_\chi'\rightarrow Y_\chi$ and $\pi'':Y_\chi''\rightarrow Y_\chi$ are isomorphic if there exists an isomorphism $\alpha:Y_\chi'\rightarrow Y_\chi''$ such that $\pi''\circ \alpha=\pi'$.
\end{definition}
We now state the \textit{geometric stratification} of the compactified Jacobian.
\begin{theorem}[Geometric stratification]\label{thm:geomstr}\cite[Proposition 3.4]{Gau97}
$\overline{\mathrm{Pic}^0}(Y_\chi)(k)$ has the following stratification:
    $$\overline{\mathrm{Pic}^0}(Y_\chi)(k)=\bigsqcup\limits_{\pi':Y_\chi'\rightarrow Y_\chi}\mathrm{Pic^0}(Y_\chi')(k),$$
    where the index runs over the isomorphism classes of partial normalizations of $Y_\chi$.
    Here, $\mathrm{Pic}^0(Y_\chi')$ is a smooth locally closed subvariety of $\overline{\mathrm{Pic}^0}(Y_\chi)$ via pushforward under $\pi':Y_\chi'\rightarrow Y_\chi$, $\mathcal{L}'\mapsto \pi'_*\mathcal{L}'$.
\end{theorem}
\begin{proof}
    Set-theoretically, the stratification of $\overline{\mathrm{Pic}^0}(Y_\chi)(k)$ is given in \cite[Proposition 3.4]{Gau97}.
Theorem \ref{thm:smoothlocus} yields that $\mathrm{Pic^0}(Y_\chi')$ is smooth and an open subscheme of $\overline{\mathrm{Pic}^0}(Y_\chi')$, which is a closed subscheme of $\overline{\mathrm{Pic}^0}(Y_\chi)$ via $\pi_*'$ by \cite[Lemma 3.1]{Bea99}.
This completes the proof. 
\end{proof}

This stratification is the same as the partition of $\overline{\mathrm{Pic^0}}(Y_\chi)(k)$ into  $\mathrm{Pic^0}(Y_\chi)(k)$-orbits as follows.
\begin{proposition}\label{prop:orbits}
    $\mathrm{Pic^0}(Y_\chi)$ acts transitively on $\mathrm{Pic^0}(Y_\chi')$ for a partial normalization $\pi':Y_\chi'\rightarrow Y_\chi$.
\end{proposition}
\begin{proof}
We first claim that the action is well-defined.
It suffices to show that $\mathrm{Pic^0}(Y_\chi)(R)$ acts on $\mathrm{Pic^0}(Y_\chi')(R)$ functorially, for a $k$-algebra $R$.
  Let $\mathcal{L}'_R$ be a line bundle on ${Y_\chi'}_{R}:={Y_\chi'}\times_k\mathrm{Spec}~R$ and let $\mathcal{L}_R$ be a line bundle on ${Y_\chi}_R:=Y_\chi\times_k\mathrm{Spec}~R$.
    Let $\pi_R':{Y_\chi'}_R\rightarrow {Y_\chi}_R$ be the base change of $\pi'$.
    By Theorem \ref{thm:geomstr}, it suffices to show that $\mathcal{L}_R\otimes_{\Mfo_{{Y_\chi}_R}}{\pi_R'}_*{\mathcal{L}'_R}$ is again a pushforward of a line bundle on ${Y_\chi'}_R$ by $\pi_R'$.
    This follows from the projection formula, which yields an isomorphism $${\pi_R'}_*\mathcal{L}'_R\otimes_{\Mfo_{{Y_\chi}_R}}\mathcal{L}_R\cong {\pi_R'}
    _*(\mathcal{L}'_R\otimes_{\Mfo_{{Y_\chi'}_R}}\pi_R'^*\mathcal{L}_R).$$
    The functoriality also follows from the projection formula.

Note that an orbit of $\mathrm{Pic^0}(Y_\chi)$ on $\mathrm{Pic^0}(Y_\chi')$ is locally closed by \cite[Lemma 2.3.3]{Spr98}.
Thus it suffices to show the transitivity of the action when $R=k$, by Hilbert's Nullstellensatz and the reducedness of $\mathrm{Pic}^0(Y_\chi')$.
The projection formula above yields that it suffices to show that the pullback map $\pi'^*:\mathrm{Pic}^0(Y_\chi)\rightarrow \mathrm{Pic}^0(Y_\chi')$ is surjective.

By \cite[Proposition 2.3]{Gau97}, $\pi'^*:\mathrm{Pic}^0(Y_\chi)\rightarrow \mathrm{Pic}^0(Y_\chi')$ is surjective as the fppf topology, so that it induces a surjective morphism of schemes.
Each point in $\mathrm{Pic}^0(Y_\chi')(k)$ is a closed point, and its fiber under $\pi'^*$ is a scheme of finite type over $k$.
Then, Hilbert's Nullstellensatz guarantees that the fiber has a closed point.
This yields that $\pi'^*:\mathrm{Pic}^0(Y_\chi)(k)\rightarrow \mathrm{Pic}^0(Y_\chi')(k)$ is surjective.
\end{proof}


\subsection{Equivalence between two stratifications of $\Phi^{-1}(\chi)$}
In this subsection, we prove in Theorem \ref{thm:eqoftwostra} that the algebraic stratification of $\Phi^{-1}(\chi)$ (cf. Corollary \ref{cor:algstraoff-1a}) and the geometric stratification of  $\overline{\mathrm{Pic}^0}(Y_\chi)(k)\left(=\Phi^{-1}(\chi)(k)\right)$  (cf. Theorem \ref{thm:geomstr}) agree, as schemes.


\subsubsection{The product formula}
From now on, fix $\gamma\in \mathfrak{gl}_n(F)$ to be the companion matrix corresponding to the characteristic polynomial $\chi$ and let $\gamma_x:=\pi_x^{d_x}\gamma$ for $x\in C$.
We rewrite the adelic description of points in $\Phi^{-1}(\chi)(k)$ in Remark \ref{rmk:adelic} to relate them with \textit{fractional ideals} of an order, following \cite[Section 3.5]{Ch14}. 
Choose $(g,\theta)(=((g_x)_{x\in C},\theta)$ in $\Phi^{-1}(\chi)(k)$.
Since $\gamma$ is elliptic regular semisimple over $F$, we may assume that $\theta=\gamma$ after applying the adjoint action by $\mathrm{GL}_n(F)$.
Note that $(g,\gamma)$ and $(g',\gamma)$ are equivalent if and only if $g$ and $g'$ are in the same orbit of the centralizer $T_\gamma(F)$ of $\gamma$, up to right action by $\mathrm{GL}_n(\Mfo)$.
Then the condition $g^{-1}\gamma g\in \pi_D^{-1}\mathfrak{gl}_n(\Mfo) $ is equivalent to $g_x^{-1}\gamma_xg_x\in \mathfrak{gl}_n(\Mfo_x)$ for each $x\in C$.
The elements $g_x\in \mathrm{GL}_n(F_x)/\mathrm{GL}_n(\Mfo_x)$ satisfying this condition are parametrized by the \textit{affine Springer fiber} $\mathfrak{X}_{\gamma_x}$, whose $R$-points are given by $$\mathfrak{X}_{\gamma_x}(R)=\{g_x\in \mathrm{GL}_n(F_x\otimes_k R)/\mathrm{GL}_n(\Mfo_x\otimes_kR)\mid g_x^{-1}\gamma_xg_x\in \mathfrak{gl}_n(\Mfo_x\otimes_k R) \}$$
for a $k$-algebra $R$.
The correspondence $((g_x)_{x\in C},\gamma)\mapsto (g_x)_{x\in C}$ therefore yields the product formula:
    \begin{proposition}\label{prop:prod}\cite[Proposition 3.5.1]{Ch14}
 The map $\Phi^{-1}(\chi)(k) \longrightarrow T_{\gamma}(F)\backslash {\prod\limits_{x\in C}}'\mathfrak{X}_{\gamma_x}(k)$, $((g_x)_{x\in C},\gamma)$ $\mapsto (g_x)_{x\in C}$ is bijective, where ${\prod\limits_{x\in C}}'$ means that $g_x=1$ for all but finitely many $x\in C$.
\end{proposition}
We now introduce the relation between the affine Springer fibers and fractional ideals, following \cite[Section 3.2.9]{Yun16}.
As shown in \cite[Proposition 2.3.1]{Ch14}, the affine Springer fiber yields an alternative description using lattices.
Namely, the map $g_x\mapsto g_x\Mfo_x^n$ gives a bijection between $\mathfrak{X}_{\gamma_x}(k)$ and the set of full lattices in $F_x^n$ stable under $\gamma_x$.
Let $R_x=\Mfo_x[X]/(\chi_x(X))$.
Then $R_x$ is an order in its total ring of fractions $E_x:=F_x[X]/(\chi_x(X))$.
It is isomorphic to the ring of matrices $\Mfo_x[\gamma_x]$, having a basis $\{1,\gamma_x,\dots,\gamma_x^{n-1}\}$ as an $\Mfo_x$-module since $\gamma_x$ is regular, which yields an $R_x$-module structure on $\Mfo_x^n$.
Thus we describe the $k$-points of the affine Springer fiber via fractional ideals.
\begin{proposition}\label{prop:frac} \cite[Lemma 3.3.1]{Yun16}
    The map $g_x\mapsto g_x\Mfo_x^n$ gives a bijection between $\mathfrak{X}_{\gamma_x}(k)$ and the set of fractional $R_x$-ideals as follows.
\end{proposition}

 If $x\notin \mathcal{S}$, then $R_x\cong \Mfo_x^n$ as rings by Proposition \ref{prop:spectralcurvegeometry} since  $\overline{\chi_x(X)}$ is assumed to be separable over $k$.
Thus both sides in the proposition are singleton. 
 We henceforth assume that $x\in\mathcal{S}$, so that $R_x$ is an integral domain. 
 Recall that $R_x$ is a Bass order (cf. Definition \ref{def:bass}) by Proposition \ref{prop:Bass}.

\begin{corollary}\label{cor:prod2}
The map $\mathrm{Pic}^0(Y_\chi)(k)\backslash\overline{\mathrm{Pic}^0}(Y_\chi)(k) \longrightarrow  {\prod\limits_{x\in C}}'\{\text{overorders of $R_x$}\}, ~~(g_x)_{x\in C}\mapsto (g_x\Mfo_x^n:g_x\Mfo_x^n)$, is bijective. 
    Here, $(g_x)_{x\in C}$ is viewed as an element in $\overline{\mathrm{Pic}^0}(Y_\chi)(k)$ using Proposition \ref{prop:prod}.
\end{corollary}
\begin{proof}

    By \cite[Proposition 4.4.3]{Yun16}, the bijection in Proposition \ref{prop:prod} yields a bijection of orbits
    $$\left.\mathrm{Pic}^0(Y_\chi)(k)\middle\backslash \overline{\mathrm{Pic}^0}(Y_\chi)(k)\right.\rightarrow \left.{\prod\limits_{x\in C}}'\left(E_x^\times/R_x^\times\right)\middle\backslash{\prod\limits_{x\in C}}'\mathfrak{X}_{\gamma_x}(k)\right.\cong {\prod\limits_{x\in C}}'\clb(R_x).$$
The second isomorphism is by Proposition \ref{prop:frac}.
The claim then follows from Proposition \ref{prop:overorder}.
\end{proof}    

In the following proposition, we describe an explicit enumeration of overorders of $R_x$. 
\begin{proposition}{\cite[Theorem 3.11]{CHL}}\label{thm:overorders}
    Write $R_x=\Mfo_x[\gamma_x]$, and $E_x=F_x[\gamma_x]$ for $x\in\mathcal{S}$. Suppose that $R_x\neq \Mfo_{E_x}$. Then we have the following classification of overorders of $R_x$.
\begin{enumerate}
    \item Suppose that $n\geq 3$.
    Then $n$ is odd, and every overorder of $R_x$ is of the form $$R_{x,r_x}=\Mfo_x[\gamma_x,\frac{\gamma_x^{r_x}}{\pi_x}], ~~~~  \textit{ where }  ~~~~~~ \frac{n+1}{2}\leq r_x\leq n. $$
    \item Suppose that $n=2$.
    Then $\ord_x(a_2)$ is odd, and every overorder of $R_x$ is of the form $$R_{x,r_x}=\Mfo_x[\frac{\gamma_x}{\pi_x^{r_x}}], ~~~~  \textit{ where }  ~~~~~~ 0\leq r_x \leq \frac{2d_x+\ord_x(a_2)-1}{2}. $$
    Here $2d_x+\ord_x(a_2)-1$ is the exponential order of $\gamma_x$ with respect to a uniformizer of $E_x$. 
\end{enumerate}
\end{proposition}
\begin{proof}
$E_x$ is a totally ramified field extension over $F_x$ by \cite[Proposition 3.5]{CHL} since $k$ is algebraically closed. Then the claim directly follows from \cite[Theorem 3.11]{CHL}.
Here, if $n<2$, then it cannot be even by \cite[Proposition 3.14]{CHL} since $R_x$ is a simple extension of $\Mfo_x$. 
\end{proof}

\begin{theorem}[Equivalence of two stratifications]\label{thm:eqoftwostra}
   Let $\pi':Y_\chi'\rightarrow Y_\chi$ be a partial normalization of $Y_\chi$.
    Then there exists a unique type $T$ of $\chi$ (cf.  Definition \ref{def:typeofchia} and Remark \ref{rmk:typeforn=2}) such that
    $$(\Phi^T)^{-1}(\chi)=\mathrm{Pic}^0(Y_\chi')$$
with respect to the identification between $\Phi^{-1}(\chi)\cong \overline{\mathrm{Pic}^0}(Y_\chi)$ in Theorem \ref{thm:spectralcorr}. 
\end{theorem}
\begin{proof}
    Since both sides are smooth and locally closed, it suffices to prove the equality on the level of $k$-points.
    Let $(g_x)_{x\in C}$ be the affine Springer fiber description of $\Phi^{-1}(\chi)(k)$, using Proposition \ref{prop:prod}.    
We claim that the tuple of overorders $((g_x\Mfo_x^n:g_x\Mfo_x^n))_{x\in C}$ of $(R_x)_{x\in C}$ is completely determined by the type $T$ of $\chi$ such that $((g_x)_{x\in C},\gamma)\in (\Phi^T)^{-1}(\chi)(k)$.
Suppose that it is true. Since each orbit of $\mathrm{Pic}^0(Y_\chi)(k)$ on $\overline{\mathrm{Pic}^0}(Y_\chi)(k)$ is completely determined by  $((g_x\Mfo_x^n:g_x\Mfo_x^n))_{x\in C}$  by Corollary \ref{cor:prod2} and is of the form $\mathrm{Pic}^0(Y_\chi')(k)$ by Proposition \ref{prop:orbits}, it completes the proof.

Since a type is defined on each $x\in \mathcal{S}$, we prove the claim by a case-by-case analysis.

(1) If $x\notin \mathcal{S}$ or if $x\in \mathcal{S}$ with $n\cdot d_{x}+\ord_{x}(a_n)=1$, then the proof of Proposition \ref{prop:Bass} yields that $R_x$ is a maximal order, and thus there are no nontrivial overorders of $R_x$.

(2)     If $x\in \mathcal{S}$ with $n=2$, then the cokernel of $g_{x}^{-1}\gamma_{x}g_{x}$ is  isomorphic to $\Mfo_{x}/(\pi_x^{a_x})\oplus \Mfo_{x}/(\pi_x^{b_x})$ such that $0\leq a_x\leq b_x$ and $a_x+b_x=2  d_x+\ord_x(a_2)$ (cf. Section \ref{subsecn=2}).
Note that $a_x$ is the biggest integer satisfying $g_{x}^{-1}\gamma_{x}g_{x}\in \pi_x^{a_x}\gl_n(\Mfo_x)$, and that the set of $a_x$'s with $x\in \mathcal{S}$ determines a type $T(=\mathcal{F})$.  
Thus it suffices to show that $a_x$  determines  the overorder $(g_{x}\Mfo_{x}^n:g_{x}\Mfo_{x}^n)$ of $R_{x}$. 
    By Proposition \ref{thm:overorders}, an overorder $R_{x}'$ of $R_{x}$ is determined by the biggest integer $r_{x}$ such that $\frac{\gamma_{x}}{\pi_{x}^{r_{x}}}\in R_{x}'$. 
 Note that 
    $$\frac{\gamma_{x}}{\pi_{x}^{r_{x}}}\in (g_{x}\Mfo_{x}^n:g_{x}\Mfo_{x}^n) ~~~~~ \textit{ if and only if } ~~~~~ \frac{1}{\pi_{x}^{r_{x}}}g_{x}^{-1}\gamma_{x}g_{x}\in \mathfrak{gl}_n(\Mfo_{x}).$$
Here the biggest integer $r_x$ is exactly the same as $a_x$.

  (3)    When $x\in \mathcal{S}$ and $n\cdot d_{x}+\ord_{x}(a_n)=2$ with $n\geq 3$, we define the nilpotency index of $\overline{g_{x}^{-1}\gamma_{x}g_{x}}$ to be the smallest positive integer $r$ such that $\left(\overline{g_{x}^{-1}\gamma_{x}g_{x}}\right)^r=0 \in \gl_n(k)$.
    Then we claim that the nilpotency index of $\overline{g_{x}^{-1}\gamma_{x}g_{x}}$ is completely determined by the type $T=(\mathcal{F}, (l_x)_{x\in \mathcal{S}^{\mathcal{F}}_1})$.

If $x\in \mathcal{S}^{\mathcal{F}}_1$  so that $\mathcal{F}_{x}\cong k\oplus k$, then the nilpotency index of $\overline{g_{x}^{-1}\gamma_{x}g_{x}}$ is $n-l_x-1$, by Remark \ref{rmk:ademfly}.
If $x\in \mathcal{S}\backslash \mathcal{S}^{\mathcal{F}}_1$ so that the cokernel of $g_{x}^{-1}\gamma_{x}g_{x}$ is  isomorphic to $\Mfo_{x}/(\pi_{x}^2)$, then $\overline{g_{x}^{-1}\gamma_{x}g_{x}}$ has rank $n-1$. 
Thus the nilpotency index of $\overline{g_{x}^{-1}\gamma_{x}g_{x}}$ is $n$.
    This proves the claim.

    Then it suffices to show that the nilpotency index of $\overline{g_{x}^{-1}\gamma_{x}g_{x}}$ determines the overorder $(g_{x}\Mfo_{x}^n:g_{x}\Mfo_{x}^n)$ of $R_{x}$.
    By Proposition \ref{thm:overorders}, an overorder $R_{x}'$ of $R_{x}$ is determined by the minimal integer $r_{x}$ such that $\frac{\gamma_{x}^{r_{x}}}{\pi_{x}}\in R_{x}'$. 
 Note that 
    $$\frac{\gamma_{x}^{r_{x}}}{\pi_{x}}\in (g_{x}\Mfo_{x}^n:g_{x}\Mfo_{x}^n) ~~~~~ \textit{ if and only if } ~~~~~ \frac{1}{\pi_{x}}g_{x}^{-1}\gamma_{x}^{r_{x}}g_{x}=\frac{1}{\pi_{x}}(g_{x}^{-1}\gamma_{x}g_{x})^{r_{x}}\in \mathfrak{gl}_n(\Mfo_{x}).$$
Thus the minimal integer $r_x$ is exactly the same as the nilpotency index of $\overline{g_{x}^{-1}\gamma_{x}g_{x}}$.
\end{proof}

\begin{remark}\label{rmk:nouse}
For a partial normalization    $\pi':Y_\chi'\rightarrow Y_\chi$,
choose $y\in Y_\chi'(k)$ and $x\in C(k)$ such that $\pi_\chi\cdot\pi'(y)=x$. Here  $\pi_\chi: Y_\chi\rightarrow C$ is the canonical projection.
If $Y_\chi'$ is also a spectral curve over $C$, then the completed local ring of $Y_\chi'$ at $y$ should be a simple extension of $\Mfo_x$ by Proposition \ref{prop:spectralcurvegeometry}.(2).
On the other hand, Proposition \ref{thm:overorders} yields that it is not necessarily simple. Thus $Y_\chi'$ is not a spectral curve in general.

Nonetheless, Theorem \ref{thm:eqoftwostra} implies that $\mathrm{Pic}^0(Y_\chi')$ can still be described within the framework of a restricted Hitchin fibration (cf. Diagram \eqref{diag:f'}), and that each point of $\mathrm{Pic}^0(Y_\chi')(k)$ has an adelic description (cf. Remark \ref{rmk:ademf} for $n=2$ and Remark \ref{rmk:ademfly} for $n\geq 3$). 
\end{remark}

\begin{remark}\label{rmk:rx}
    The proof of Theorem \ref{thm:eqoftwostra} gives a a precise description of overorders $(R_{x,r_x})_{x\in \mathcal{S}}$ corresponding to a given type $T=(\mathcal{F},(l_x)_{x\in \mathcal{S}^{\mathcal{F}}_1})$ of $\chi$ (cf.  Definition \ref{def:typeofchia} and Remark \ref{rmk:typeforn=2}). In this remark, we will describe both a type and an overorder.  
    These will be used in Section \ref{sec:dec}.
    
    \begin{enumerate}
 \item The sheaf $\mathcal{F}$ is completely determined by $\mathcal{F}_x$'s for $x\in \mathcal{S}$ (cf. the paragraph following Definition \ref{def:coherentsheafF}). Let 
 \begin{equation}\label{eq:defofs'}
  \mathcal{S}':=\{x\in \mathcal{S}|n\cdot d_x+\ord_x(a_n)>1\}.    
 \end{equation}
 Note that $\mathcal{S}'=\{x\in \mathcal{S}|n\cdot d_x+\ord_x(a_n)=2\}$ if $n\geq 3$. 
 Then, for $x\in \mathcal{S}\setminus \mathcal{S}'$, we have $\mathcal{F}_x\cong k$ and $R_x=\Mfo_{E_x}$. 
 Since $\mathcal{F}_x$ with $x\in \mathcal{S}'$ for $n\geq 3$ is either $k\oplus k$ or $\Mfo_x/(\pi_x^2)$,  $$\textit{the number of $\mathcal{F}$'s}=\begin{cases}
2^{\#\mathcal{S}'}&  \textit{ if $n\geq 3$};\\\prod\limits_{x\in \mathcal{S'}}(\lfloor\frac{2d_x+\ord_x(a_2)}{2}\rfloor+1)
& \textit{ if $n=2$}.
    \end{cases}$$ 
    
Here $\mathcal{S}^{\mathcal{F}}_1$ depends on $\mathcal{F}$ but is contained in $\mathcal{S}'$ for any $\mathcal{F}$.
We also have
$$\textit{the number of types}=\begin{cases}
\left(\frac{n+1}{2}\right)^{\#\mathcal{S}'}&  \textit{ if $n\geq 3$};\\
\prod\limits_{x\in \mathcal{S'}}(\lfloor\frac{2d_x+\ord_x(a_2)}{2}\rfloor+1)
& \textit{ if $n=2$}.
    \end{cases}$$

    \item We will explain a precise description of an overorder of $R_x$ given in Proposition \ref{thm:overorders}, in terms of a type $T$.  
   \begin{enumerate}
       \item     If $x\in \mathcal{S}\backslash\mathcal{S}'$ then $R_x=\Mfo_{E_x}$, the maximal order. 
    
    \item For $n\geq 3$, if  $x\in \mathcal{S}^{\mathcal{F}}_1$ then $r_x=n-l_x-1$ and if $x\in \mathcal{S}'\backslash \mathcal{S}^{\mathcal{F}}_1$, then $r_x=n$.
    

    \item
    If $n=2$ and  $x\in \mathcal{S}'$, then $r_x=a_x$, where $\mathcal{F}_x\cong \Mfo_x/(\pi_x^{a_x})\oplus \Mfo_x/(\pi_x^{b_x})$ with $a_x\leq b_x$ and $a_x+b_x=2d_x+\ord_x(a_2)$ (cf. Section \ref{subsecn=2}).
   \end{enumerate}
    \end{enumerate}
\end{remark}



\section{Decomposition of $H_c^i(\Phi^{-1}(\chi),\mathbb{Q}_\ell)$}\label{sec:dec}
The goal of this section is Theorem \ref{thm:cohomology}, which describes the $\ell$-adic cohomology of the Hitchin fiber (equivalently $\overline{\mathrm{Pic}^0}(Y_\chi)$ by Theorem \ref{thm:spectralcorr}) under a certain restriction described below.

In this section let $k=\mathbb{F}_q$ and we maintain Setting \ref{settings} over $\bar{k}$ and the situation in Proposition \ref{prop:Bass} over $\bar{k}$. 
Thus $Y_\chi$ is geometrically integral,  $x\in \mathcal{S}$ is defined over $k$,  and  $R_x\left(=\Mfo_x[X]/(\chi_x(X))\right)$ for $x\in \mathcal{S}\left(\subset C(k)\right)$ is a Bass order in $E_x\left(=F_x[X]/(\chi_x(X))\right)$ which is a totally ramified field extension of $F_x$. 
This also yields that any singularity of $Y_{\chi}\times_k\bar{k}$ lies in $Y_\chi(k)$ and that the canonical projection $\pi_\chi: Y_\chi\rightarrow C$ yields a bijection between the set of singularities in $Y_\chi(\bar{k})$ and $\mathcal{S}'$ (cf. \eqref{eq:defofs'}).
Thus we identify them.
Since this is important in Theorem \ref{thm:cohomology}, we emphasize it below:
\begin{equation}\label{eq:emphas'}
\pi_\chi:  \{\textit{singularities of }Y_\chi(\bar{k})\}\left(\subset Y_\chi(k)\right)   \xrightarrow{\sim} \mathcal{S}'.
\end{equation}
We also note that $Y_\chi$ has double singularities and that its local rings are integral domains(cf. Proposition \ref{prop:Bass}).

Then the classification of overorders of $R_x$ described in Proposition \ref{thm:overorders} also holds. 
In addition, the number of overorders of $R_x$ is the same as that of $R_x\otimes_k\bar{k}$.
This yields that any partial normalization of $Y_\chi\times_k\bar{k}$ is defined over $k$ since a partial normalization corresponds to a tuple of overorders of $R_x$'s with $x\in \mathcal{S}$. 
In the following, we provide a different proof of Theorem \ref{thm:geomstr} over $k$ to avoid the assumption that $k=\bar{k}$. For a field extension $k'/k$ in $\bar{k}$, let $Y_{\chi,k'}:=Y_\chi\times_kk'$. 
\begin{proposition}[Geometric stratification over $k$]\label{prop:geomstrfork0}
    $\overline{\mathrm{Pic}^0}(Y_{\chi})(k')$ has the following stratification:
    $$\overline{\mathrm{Pic}^0}(Y_{\chi})(k')=\bigsqcup\limits_{\pi':Y_{\chi}'\rightarrow Y_{\chi}}\mathrm{Pic^0}(Y_{\chi}')(k') ~~~~  \textit{ for a finite field extension $k'$ over $k$},$$
    where $\pi'$ runs over the isomorphism classes of partial normalizations of $Y_\chi$ over $k$.
\end{proposition}
\begin{proof}
For $\mathcal{L}\in \overline{\mathrm{Pic}^0}(Y_{\chi})(k')$ let $Y_{\chi}':=\underline{\mathrm{Spec}}(\mathcal{H}\mathit{om}_{\Mfo_{Y_\chi,k'}}(\mathcal{L},\mathcal{L}))$.  
 \cite[The second paragraph of Section 2]{Bea99} yields that 
  $\pi': Y_{\chi}'\rightarrow Y_{\chi,k'}$  is 
  a partial normalization of $Y_{\chi,k'}$ defined over $k'$  such that $\mathcal{L}$ is a rank $1$ torsion-free $\Mfo_{Y_{\chi}'}$-sheaf. We write $\mathcal{L}'$ to stand for $\mathcal{L}$ as an $\Mfo_{Y_{\chi}'}$-sheaf. 
  Note that $\pi'$ is defined over $k$ as mentioned just before the proposition.
 We claim that $\mathcal{L}'$ is locally free of rank $1$.

    For each $y'\in Y_{\chi}'(\bar{k})$ with $y:=\pi'(y')\in Y_{\chi,k'}(\bar{k})$, the local completion of $Y_{\chi}'$ at $y’$ is a Gorenstein ring since it is  an overorder of $\widehat{\Mfo}_{Y_{\chi,y}}$ which is a Bass order by Proposition \ref{prop:Bass} (cf. \cite[pages 96-97]{Lam99}).
    Therefore, $Y_{\chi}'$ is a Gorenstein curve.
Since $\Mfo_{Y_\chi'}\left(=\mathcal{H}\mathit{om}_{\Mfo_{Y_\chi,k'}}(\mathcal{L},\mathcal{L})\right)\cong \mathcal{H}\mathit{om}_{\Mfo_{Y_\chi'}}(\mathcal{L}',\mathcal{L}')$, \cite[Theorem 3.1]{Vas68} yields that $\mathcal{L}'$ is a locally free $\Mfo_{Y_\chi'}$-sheaf of rank $1$.
In addition, the fact that $\Mfo_{Y_\chi'}=\mathcal{H}\mathit{om}_{\Mfo_{Y_\chi,k'}}(\mathcal{L},\mathcal{L})$ yields that the right hand side is a disjoint union.
\end{proof}

From now on, we assume that the normalization of $Y_\chi$ is isomorphic to $\mathbb{P}_k^1$.
This assumption yields that each algebraic stratum is isomorphic to affine space, as seen in the following lemma.
\begin{lemma}\label{lem:affine}
For a type $T=(\mathcal{F}, (l_x)_{x\in \mathcal{S}^{\mathcal{F}}_1})$ of $\chi$, $(\Phi^T)^{-1}(\chi)$ is isomorphic to an affine space over $k$ of dimension
  $$ \begin{cases}
\sum\limits_{x\in \mathcal{S}^{\mathcal{F}}_1}(\frac{n-3}{2}-l_x)+\#(\mathcal{S}'\setminus\mathcal{S}^{\mathcal{F}}_1)\cdot\frac{n-1}{2}& \textit{ if $n\geq 3$};\\
\sum\limits_{x\in\mathcal{S}'}\frac{b_x-a_x-1}{2}=\sum\limits_{x\in \mathcal{S}'}(\frac{2d_x+\ord_x(a_2)-1}{2}-a_x)& \textit{ if $n=2$}.
    \end{cases}$$
Here,  $\mathcal{F}_x\cong \Mfo_x/(\pi_x^{a_x})\oplus \Mfo_x/(\pi_x^{b_x})$ with $a_x+b_x=2d_x+\ord_x(a_2)$ and $0\leq a_x\leq b_x$ when $n=2$.

\end{lemma}
\begin{proof}
    Using Theorem \ref{thm:eqoftwostra}, consider a partial normalization $\pi':Y_\chi'\rightarrow Y_\chi$ corresponding to $T$.
    As discussed in the beginning of this section, $Y_\chi'$ corresponds bijectively to a tuple of overorders of $R_x$ for each $x\in \mathcal{S}$.
    Following Proposition \ref{prop:overorder} and Remark \ref{rmk:rx}, for $x\in\mathcal{S}'$, let $R_{x,r_x}$ be this overorder.
    
    By \cite[Exercise II.6.9.(a)]{Har77}, we have an exact sequence
    \begin{equation}\label{eq:fppf}
            0\rightarrow \prod\limits_{x\in \mathcal{S}}\left(\Mfo_{E_{x}}\right)^\times/(R_{x,r_x})^\times\rightarrow \mathrm{Pic}(Y_{\chi}')\rightarrow \mathrm{Pic}(\mathbb{P}^1_{k})\rightarrow 0 ~~~ \textit{  as fppf sheaves of groups on $\mathrm{Spec}~k$,}
                \end{equation}
    where the right map is the pullback of the normalization map. 
    We have $\mathrm{Pic}(\mathbb{P}^1_{k})(k')\cong \mathbb{Z}$ via the degree map.
    Since pullback preserves the degree, $\mathrm{Pic}(Y_{\chi}')\cong \mathbb{Z}\times \mathrm{Pic}^0(Y_{\chi}')$ so that $\prod\limits_{x\in \mathcal{S}}\left(\Mfo_{E_{x}}\right)^\times/(R_{x,r_x})^\times\cong \mathrm{Pic}^0(Y_{\chi}')$, which is smooth over $k$ (cf. Remark \ref{rmk:jac}.(2)). 

    Let $k'=\mathbb{F}_{q^d}$ so that $[k':k]=d$.
    We count the number of $k'$-points of $\mathrm{Pic}^0(Y_\chi')$ using the above isomorphism.
    By \cite[Equation 3.2]{CHL}, 
    $$\prod\limits_{x\in \mathcal{S}}\left(\#\left(\Mfo_{E_{x}}\otimes_{k} k'\right)^\times/(R_{x,r_x}\otimes_{k} k')^\times\right)= 
    \begin{cases}
        q^{\sum\limits_{x\in \mathcal{S}'}d(r_x-\frac{n+1}{2})}& \textit{ if $n\geq 3$};\\
        q^{\sum\limits_{x\in \mathcal{S}'}d(\frac{2d_x+\ord_x(a_2)-1}{2}-r_x)}& \textit{ if $n=2$}.
    \end{cases}$$
   Note that the above holds for any positive integer $d$. 
    Thus the structural theorem of linear algebraic groups yields that $\left(\Mfo_{E_{x}}\right)^\times/(R_{x,r_x})^\times$  is represented by $\mathbb{G}_a^{\sum\limits_{x\in \mathcal{S}'}(r_x-\frac{n+1}{2})}$ (if $n\geq 3$) or $\mathbb{G}_a^{\sum\limits_{x\in \mathcal{S}'}(\frac{2d_x+\ord_x(a_2)-1}{2}-r_x)}$ (if $n=2$), since it is commutative.
    Then Remark \ref{rmk:rx}.(2) yields the desired dimension formula. 
\end{proof}
\begin{theorem}\label{thm:cohomology}
    Let $\mathcal{I}$ be the set of types of $\chi$.
 Then 
 $$H_c^i(\Phi^{-1}(\chi)_{\bar{k}},\mathbb{Q}_\ell)\left(\cong  H_c^i(\overline{\mathrm{Pic}^0}(Y_\chi)_{\bar{k}},\mathbb{Q}_\ell)\right) \cong \bigoplus \limits_{T\in \mathcal{I}}H_c^i((\Phi^T)^{-1}(\chi)_{\bar{k}},\mathbb{Q}_\ell)$$
 as a $\mathrm{Gal}(\bar{k}/k)$-module. For a type $T=(\mathcal{F},(l_x)_{x\in \mathcal{S}^{\mathcal{F}}_1})$,
    $$H_c^i((\Phi^T)^{-1}(\chi)_{\bar{k}},\mathbb{Q}_\ell)\cong \begin{cases}
        \mathbb{Q}_\ell(-\frac{i}{2})&\textit{ if $n\geq 3$ and $i=2(\sum\limits_{x\in \mathcal{S}'}(\frac{n-3}{2}-l_x))$};\\
        \mathbb{Q}_\ell(-\frac{i}{2})&\textit{ if $n=2$ and $i=\sum\limits_{x\in\mathcal{S}'}(b_x-a_x-1)=2(\sum\limits_{x\in\mathcal{S}'}(\frac{2d_x+\ord_x(a_2)-1}{2}-a_x))$};\\
        0&otherwise,
    \end{cases}$$
where   we let $l_x:=-1$ if $x\in \mathcal{S}'\backslash\mathcal{S}^{\mathcal{F}}_1$ with $n\geq 3$. Here
\begin{itemize}
\item $\mathcal{S}'$ is given in \eqref{eq:defofs'} and \eqref{eq:emphas'} (see Setting \ref{settings}.(2) for $\mathcal{S}$), $\mathcal{S}_1^\mathcal{F}$ is given in Definition \ref{def:typeofchia}.(2);
\item $\mathcal{F}$ is described in Remark \ref{rmk:rx}.(1) (cf. Definition \ref{def:coherentsheafF});
\item when $n\geq 3$,  $n$ is odd (cf. Proposition \ref{thm:overorders}.(1)) and $l_x$ for $x\in \mathcal{S}_1^\mathcal{F}$ is an integer such that $0\leq l_x \leq n/2-1$ (cf. Definition \ref{def:typeofchia});
    \item when $n=2$,  $\ord_x(a_2)$ is odd (cf. Proposition \ref{thm:overorders}.(2)) and    $0\leq a_x\leq b_x \in \mathbb{Z}$   such that $a_x+b_x=2d_x+\ord_x(a_2)$ (cf. Lemma \ref{lem:affine}).    
\end{itemize}
\end{theorem}
\begin{proof}
    The second statement is a direct consequence of Lemma \ref{lem:affine}.
    By Lemma \ref{lem:affine}, the algebraic stratification of $\Phi^{-1}(\chi)$ yields an affine paving over $k$.
    Then the decomposition of $H_c^i(\Phi^{-1}(\chi)_{\bar{k}},\mathbb{Q}_\ell)$ as $\mathbb{Q}_l$-vector spaces follows from \cite[Section 12.2]{Jan04}, and the verification of this decomposition as a $\mathrm{Gal}(\bar{k}/k)$-module is from \cite[beginning of Section 2]{RS77}.
\end{proof}

Using Remark \ref{rmk:rx}.(1), we describe a combinatorially closed formula of the above theorem.

\begin{corollary}\label{cor:explicitcohomology}
Let $\mathcal{S}'=\{x_1, \dots, x_t\}$ so that $\#\mathcal{S}'=t$, the number of singularities of $Y_\chi$ (cf. \eqref{eq:emphas'}).  
\[
\textit{Let }  \begin{cases}
        n=2m+1, ~P_m(X)=1 + X + X^2 + \dots + X^m  &\textit{ if $n\geq 3$};\\
       2d_{x_j}+\ord_{x_j}(a_2)=1+2m_j\left(\geq 3\right), ~Q_j(X) = 1 + X + X^2 + \dots + X^{m_j} &\textit{ if $n=2$}. 
    \end{cases}
\]

\begin{enumerate}
    \item 
Then $H_c^{2i+1}(\Phi^{-1}(\chi)_{\bar{k}},\mathbb{Q}_\ell)=0$ for any $i$ and $H_c^{2i}(\Phi^{-1}(\chi)_{\bar{k}},\mathbb{Q}_\ell)$ is formulated below:
\[
H_c^{2i}(\Phi^{-1}(\chi)_{\bar{k}},\mathbb{Q}_\ell)\left(\cong  H_c^{2i}(\overline{\mathrm{Pic}^0}(Y_\chi)_{\bar{k}},\mathbb{Q}_\ell)\right) \cong
 \begin{cases}
        \left(\mathbb{Q}_\ell(-i)\right)^{P(i,m,t)}&\textit{ if $n\geq 3$ and $0\leq i \leq mt$};\\
        0&\textit{ if $n\geq 3$ and $i>mt$};\\
        \left(\mathbb{Q}_\ell(-i)\right)^{Q(i,m_1,\dots, m_t)}&\textit{ if $n=2$ and $i\leq\sum{m_j}$};\\
        0&\textit{ if $n=2$ and $i>\sum{m_j}$},
    \end{cases}
\]
\[
\textit{where }  \begin{cases}
      P(i,m,t)=\sum\limits_{k=0}^{\lfloor \frac{i}{m+1} \rfloor} (-1)^k \binom{t}{k} \binom{i - k(m+1) + t - 1}{t - 1}=\textit{the coefficient of }X^i \textit{ in } \left(P_m(X)\right)^t;\\
    Q(i,m_1,\dots,m_t)=\sum\limits_{A \subseteq \{1, \dots, t\}} (-1)^{\#A} \binom{i - M_A + t - 1}{t - 1}=\textit{the coefficient of }X^i \textit{ in } Q_1(X) Q_2(X) \dots Q_t(X).
    \end{cases}
\]
Here $M_A = \sum_{j \in A} (m_j + 1)$ and the binomial coefficient $\binom{n}{k}$ is defined to be $0$ if $n < k$.    

\item 
The Poincar\'e polynomial is $\begin{cases}
\left(P_m(X^2)\right)^t  &\textit{ if $n\geq 3$};\\
Q_1(X^2) Q_2(X^2) \dots Q_t(X^2) &\textit{ if $n=2$}. 
    \end{cases}$

\end{enumerate}

\end{corollary}

\begin{proof}
Theorem \ref{thm:cohomology} directly yields that the $(2i+1)$-th Betti number is zero.  
By Remark \ref{rmk:rx}.(1), the $2i$-th Betti number is formulated as follows:
\begin{itemize}
    \item when $n\geq 3$, it is the number of integer solutions to the equation
    $Y_1 + Y_2 + \dots + Y_t = i$
subject to the constraint
    $0 \le Y_j \le m \quad \text{for all } j = 1, \dots, t$;

    \item when $n=2$, it is the number of integer solutions to the equation
    $Y_1 + Y_2 + \dots + Y_t = i$
subject to the constraints
$0 \le Y_j \le m_j \quad \text{for each } j = 1, 2, \dots, t$.    
\end{itemize}
The first (resp. the second) is the same as the coefficient of $X^i$ in $\left(P_m(X)\right)^t$ (resp. $Q_1(X) Q_2(X) \dots Q_t(X)$). These formulas are well-known, which completes the proof. 
\end{proof}


    

\begin{remark}\label{rmk:kun}
\begin{enumerate}
\item
Theorem \ref{thm:cohomology} yields that     $H_c^i(\Phi^{-1}(\chi)_{\bar{k}},\mathbb{Q}_\ell)$ is pure. 
    That is, the eigenvalues of the action of the geometric Frobenius  $\mathrm{Frob}_{q^d}\in \mathrm{Gal}(\bar{k}/k)$ on $H_c^i(\Phi^{-1}(\chi)_{\bar{k}},\mathbb{Q}_\ell)$  have absolute value $q^{\frac{id}{2}}$ in $\mathbb{C}$, regardless of the choice of an isomorphism between $\overline{\mathbb{Q}_\ell}$ and $\mathbb{C}$.  
The purity also follows from \cite[Theorem 1]{Che24} which states the purity of $H^i_c(\mathfrak{X}_{\gamma_x,\bar{k}},\mathbb{Q}_\ell)$ for each $x$.
    
    By \cite[Section 2]{Che24}, we have $\mathfrak{X}_{\gamma_x}\cong \mathfrak{X}_{\gamma_x}^0\times \mathbb{Z}$, where $\mathfrak{X}_{\gamma_x}^0(k)$ parametrizes lattices $L\subset F_x^n$ such that $[L:\Mfo_x^n\cap L]=[\Mfo_x^n:\Mfo_x^n\cap L]$.
    \cite[page 519]{Lau06} shows that $\mathfrak{X}_{\gamma_x}^0$ is isomorphic to the quotient of $\mathfrak{X}_{\gamma_x}$ by the action of an infinite cyclic group $E_x^\times/\Mfo_{E_x}^\times$ generated by a uniformizer of $E_x$.
    Since $E_x^\times/R_x^\times=(E_x^\times/\Mfo_{E_x}^\times)\times (\Mfo_{E_x}^\times/R_x^\times)\cong \mathbb{Z}\times (\Mfo_{E_x}^\times/R_x^\times)$ and since $\mathrm{Pic}^0(Y_\chi)$ is identified with $\prod\limits_{x\in\mathcal{S}}(\Mfo_{E_x}^\times/R_x^\times)$ by the paragraph following  \eqref{eq:fppf}, the following homeomorphism 
    $$\mathrm{Pic}^0(Y_\chi)\times^{\prod\limits_{x\in \mathcal{S}}(E_x^\times/R_x^\times)}\prod\limits_{x\in \mathcal{S}'}\mathfrak{X}_{\gamma_x}\rightarrow \overline{\mathrm{Pic}^0}(Y_\chi) ~~~~   \textit{  in \cite[Proposition 4.4.3]{Yun16}}$$
     yields that $\Phi^{-1}(\chi)\left(\cong \overline{\mathrm{Pic}^0}(Y_\chi)\right)\cong\prod\limits_{x\in \mathcal{S}'}\mathfrak{X}_{\gamma_x}^0$, so that $H^i_c(\Phi^{-1}(\chi)_{\bar{k}},\mathbb{Q}_\ell)$ is pure.
\item
When $\mathcal{S}'$ is a singleton $\{x\}$, Corollary \ref{cor:explicitcohomology} combined with the above (1) yields that 
$$H_c^i(\overline{\mathrm{Pic}^0}(Y_\chi)_{\bar{k}},\mathbb{Q}_\ell)\cong H_c^i(\mathfrak{X}_{\gamma_x,\bar{k}}^0,\mathbb{Q}_\ell)\cong H_c^i(\mathbb{P}_{\bar{k}}^{m_x},\mathbb{Q}_\ell).$$

For general $\mathcal{S}'$, by using the above formula for each $H_c^i(\mathfrak{X}_{\gamma_{x_j},\bar{k}}^0,\mathbb{Q}_\ell)$, we have the following observation:
$$H_c^{i}(\Phi^{-1}(\chi)_{\bar{k}},\mathbb{Q}_\ell)\cong H_c^i(\overline{\mathrm{Pic}^0}(Y_\chi)_{\bar{k}},\mathbb{Q}_\ell)\cong H_c^i(\prod\limits_{x\in \mathcal{S}'}\mathbb{P}_{\bar{k}}^{m_x},\mathbb{Q}_\ell), \textit{ where } m_x=\begin{cases}
    \frac{n-1}{2}&\textit{ if }n\geq3;\\
    \frac{2d_x+\ord_x(a_2)-1}{2}&\textit{ if }n=2.
\end{cases}$$
By this equation, 
$H_c^i(\Phi^{-1}(\chi)_{\bar{k}},\mathbb{Q}_\ell)$ is understood via the Künneth formula.
\end{enumerate}
\end{remark}

\bibliographystyle{alpha}
\bibliography{References}
\end{document}